\documentclass{article}
\usepackage{color}
\usepackage{graphicx}
\usepackage{amsmath}
\usepackage{amssymb}
\usepackage{latexsym}
\usepackage{yhmath}
\usepackage{mathrsfs}
\usepackage[numbers,sort&compress]{natbib}
\usepackage{amsthm}
\usepackage{pgf}
\usepackage{graphicx}
\usepackage{tikz}
\usepackage{enumitem}
\usepackage{subcaption}

\usepackage{graphicx}
\usepackage[colorlinks,linkcolor=blue,anchorcolor=red,citecolor=green]{hyperref}
\usepackage[T1]{fontenc}
\usepackage[english]{babel}
\usepackage{listings}
\usepackage{xcolor,mathrsfs,url}
\usepackage{amssymb}
\usepackage{amsmath}
\usepackage{ifthen}
\usepackage{caption}
\usepackage{float}
\usepackage{authblk}

{\LARGE }

\newcommand{\ddddd}{\mathrm{d}}

\newcommand{\R}{\mathrm{Re}}
\newcommand{\deff}{\overset{\mathrm{def}}{=}}
\newcommand{\pp}{\mathcal{P}}
\newcommand{\bb}{\mathcal{B}}
\newcommand{\oo}{\mathcal{O}}
\newcommand{\re}{\mathrm{Re}}
\newcommand{\im}{\mathrm{Im}}

\newcommand*\circled[1]{\tikz[baseline=(char.base)]{
		\node[shape=circle,draw,inner sep=0.3pt] (char) {#1};}}

\begin{document}
	\baselineskip=17pt
	
	\title{Long-time asymptotic analysis for defocusing Ablowitz-Ladik system with initial value in lower regularity}
	\author[a,b]{Meisen Chen\thanks{Corresponding author: chenms93@szu.edu.cn}}
	\author[c]{Engui Fan}
	\author[a]{Jingsong He}
	\affil[a]{\small Institute for Advanced Study, Shenzhen University, Shenzhen, 518060, China}
	\affil[b]{College of Physics and Optoelectronic Engineering, Shenzhen University, Shenzhen, 518060, China}
	\affil[c]{School of Mathematical Sciences, Fudan University, Shanghai, 200433, China}
	\maketitle
	
	\theoremstyle{plain}
	\newtheorem{proposition}{Proposition}[section]
	\newtheorem{lemma}[proposition]{Lemma}
	\newtheorem{theorem}[proposition]{Theorem}
	\newtheorem{drhp}[proposition]{$\bar\partial$-RH problem}
	\newtheorem{srhp}[proposition]{Scalar RH problem}
	\newtheorem{rhp}[proposition]{RH problem}
	\newtheorem{dbarproblem}[proposition]{$\bar\partial$-problem}
	\newtheorem{remark}[proposition]{Remark}
	
	\begin{abstract}
		Recently, we have given the $l^2$ bijectivity for defocusing Ablowitz-Ladik systems in the discrete Sobolev space $l^{2,1}$ by inverse spectral method. 
		Based on these results, the goal of this article is to investigate the long-time asymptotic property for the initial-valued problem of the defocusing Ablowitz-Ladik system with initial potential in lower regularity.
		The main idea is to perform proper deformations and analysis to the corespondent Riemann-Hilbert problem with the unit circle as the jump contour $\Sigma$. 
		As a result, we show that when $|\frac{n}{2t}|\le 1<1$, the solution admits Zakharov-Manakov type formula, and when $|\frac{n}{2t}|\ge 1>1$, the solution decays fast to zero. 
		\\
		
		\noindent\textbf{Key Words}: defocusing Ablowitz-Ladik system, inverse spectral method, long-time asymptotic property, Riemann-Hilbert problem\\
		\noindent\textbf{2010 Mathematics Subject Classification Numbers:} 37K15, 35Q15, 35Q55

	\end{abstract}

	\tableofcontents
	\addtocontents{toc}{\protect\setcounter{tocdepth}{1}}
	
\section{Introduction}
\indent

It's well known that inverse scattering transform is an effective method when solving the integrable system. 
In 1967, the inverse scattering transform is firstly introduced when solving the KdV equation by Gardne, Greene, Kruskal and Miura \cite{gardner1967method}.
This method is also applied to ZS-AKNS systems \cite{Ablowitz1974the}. 
More literatures for solving the continuous and discrete integrable systems refers to \cite{Wang2010integrable,Biondini2014inverse,Biondini2015inverse,Kraus2015the,Pichler2017on,xu2014a,xiao2016a,kang2018multi,yang2019high,yang2010nonlinear,yang2021riemann,teschl1999inverse,Ablowitz2007inverse,ablowitz2020discrete,Prinari,Ortiz2019inverse,chen2021riemann},
Except for solving initial-valued problem, this method also have a great number of important results in mathematics and physics. 
Particularly, Deift and Zhou apply this method to the nonlinear Schr\"odinger equations \cite{deift1994long,deift1994long2} and modified KdV equations \cite{deift1993steepest} to obtain solutions, and further develop a nonlinear steepest descent method to study the long-time asymptotic analysis with potentials in Schwartz space. 
This method also has been applied to numerous integrable systems for the long-time asymptotic analysis \cite{grunert2009long,de2009long,xu2015long,xu2018long,huang2015long,zhu2018the,kruger2009long,kruger2009long2,yamane2015long,yamane2019long,chen2020long}. 
Recent years, people become interested in extending the long-time asymptotic analysis for the integrable system with lower regularity. 
In particular, for nonlinear Schr\"odinger equation as one of the most important integrable systems, based on the $L^2$-Sobolev bijectivity for the inverse scattering transform \cite{zhou1998l2}, with a dbar steepest descent method, people investigate the long-time asymptotic property when the initial potential belongs to a weighted Sobolev space:
\begin{align*}
	H^{1,1}=\{f\in L^2(\mathbb{R}):xf,f',xf'\in L^2(\mathbb{R})\}. 
\end{align*}

In this paper, we focus on defocusing Ablowitz-Ladik systems
\begin{align}\label{e1}
	i\partial_tq_n(t)=q_{n+1}(t)-2q_n(t)+q_{n-1}(t)-|q_n(t)|^2(q_{n+1}(t)+q_{n-1}(t)),
\end{align}
that is integrable systems introduced by Ablowitz and Ladik \cite{Ablowitz1975nonlinear,ablowitz1976nonlinear} in 1975-1976, where $n\in\mathbb{Z}$ is the discrete spatial variable and $t\in\mathbb{R}^+$ is the continuous variable. 
It's shown as the spatial integrable discretization of the defocusing nonlinear Schr\"odinger equation:
\begin{align*}
	iu_t+u_{xx}-2|u|^2u=0,
\end{align*}
and there are many important researches for it in aspects of mathematics and physics. 
In 2005, Nenciu \cite{nenciu2005lax} constructs the Lax pair for defocusing Ablowitz-Ladik systems by the connection between defocusing Ablowitz-Ladik systems and the orthogonal polynomials on the unit circle. 
In 2006, Chow et al \cite{chow2006analytic} show the analytic doubly periodic waves pattern for Ablowitz-Ladik systems. 
Using methods of algebraic geometry, Miller et al \cite{miller1995finite} obtain the finite genus solutions for Ablowtiz-Ladik systems.
Recently, Yamane \cite{yamane2014long} studies the long-time asymptotic behavior by Deift-Zhou method with initial data bounded by $\sup_{n\in\mathbb{Z}}|q_n|<1$ and in the discrete Schwartz space:
\begin{align*}
	\{q_n\}_{n=-\infty}^\infty\in\left\{\{a_n\}_{n=-\infty}^\infty:\sum_{n=-\infty}^{\infty}(1+n^2)^ka_n^2<\infty,\ k\in\mathbb{N}^+\right\}. 
\end{align*}
Here, we present the long-time asymptotic analysis for initial-valued problem of (\ref{e1}) with initial potential in lower regularity, and it satisfies
\begin{align}\label{e2}
	&\{q_n\}_{-\infty}^\infty\in l^{2,1}=\left\{\{a_n\}_{-\infty}^\infty:\sum_{n=-\infty}^{\infty}(1+n^2)a_n^2<\infty\right\},\quad \sup_{n\in\mathbb{Z}}|q_n|<1, 
\end{align}
where $l^{2,1}$ is a discrete weighted Sobolev space. 

For the initial potential satisfies (\ref{e2}), recently,  we prove in \cite{chen2022sobolev} that the direct scattering mapping maps these potentials to reflection coefficients $r(\lambda)$ which belongs to a Sobolev space $H^1_\theta(\Sigma)$ and is bounded by $\parallel{r}\parallel_{L^\infty(\Sigma)}<1$, where $\Sigma$ is the jump contour shown in Figure \ref{f1s} and $\theta\in[0,2\pi]$ is the parameter of it. 
Reversely, by the inverse scattering mapping, if reflection coefficients belong to $H^1_\theta(\Sigma)$ and is bounded by $\parallel{r}\parallel_{L^\infty(\Sigma)}<1$, potentials also belongs to the discrete weighted Sobolev space $l^{2,1}$. 
In fact, by the argument in \cite{chen2022sobolev}, these two mappings is of Lipschitz continuity.
Moreover, for Ablowitz-Ladik systems, if we denote $r(\lambda,t)$ the reflection coefficient for $q_{n}(t)$,  the time flow
\begin{align*}
	r(\lambda)\equiv r(\lambda,0)\mapsto r(\lambda,t)=r(\lambda)e^{2i(\cos\theta-1)t},\quad \lambda=e^{i\theta}\in\Sigma, 
\end{align*}
persists the reflection coefficient in $H^1_\theta(\Sigma)$ and bounded by $\parallel r(\cdot,t)\parallel_{L^\infty(\Sigma)}<1$. 
In (\ref{e2}), the weighted Sobolev space $l^{2,1}$ is the minimal condition for the solution solved by the inverse scattering transform. 

\begin{figure}
	\centering\includegraphics[width=0.3\linewidth]{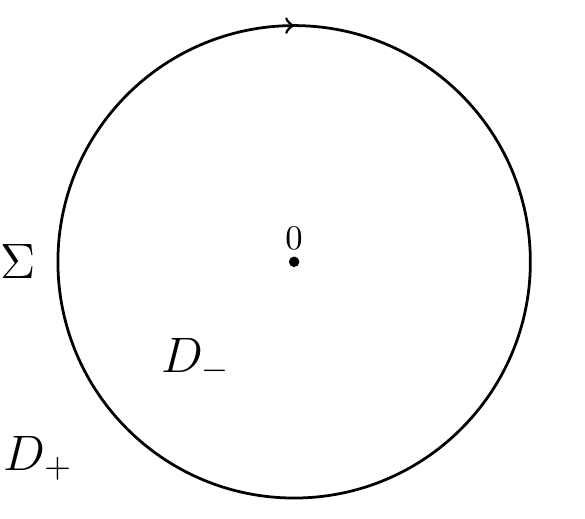}
	\caption{$\Sigma$: jump contour for RH problem \ref{r2.1}} $D_+/D_-$: the region outside/inside of $\Sigma$.\label{f1s}
\end{figure}

Since $r\in H^1_\theta(\Sigma)$, we obtain Riemann-Hilbert (RH) problem \ref{r2.1} and it couldn't applied proper rational approximation to $r(\lambda)$; therefore, it's hard to extend the jump contour and make the RH transform like that in \cite{yamane2014long}. 
For $r\in H^1_\theta(\Sigma)$, by Sobolev embedding theorem, it's also $\frac{1}{2}$-H\"older continuous on $\Sigma$;
as a result, we deform RH problem \ref{r2.1} properly into a $\bar\partial$-RH problem;
moreover, solution for this $\bar\partial$-RH problem can be factorized into a product of solutions for an RH problem and a $\bar\partial$ problem;
and, we further analyze this two problems separately and finally obtain the long-time asymptotic property.
This idea is generalized from the dbar steepest descent method that people applied to orthogonal polynomials and nonlinear Schr\"odinger equations \cite{dieng2008long,Borghese2018long}.
This method make it eligible to analyze the long-time asymptotic property under the condition $r\in H^1_\theta(\Sigma)$. 
In our method, the jump contour is a unit circle and there are two separate first-order stationary phase points, which make the deformation of RH problem more tricky.

In this paper, for Hilbert spaces $P_1$ and $P_2$, we denote $\bb(P_1,P_2)$ as the normed linear space consisting of all linear bounded operators from $P_1$ to $P_2$, and simply write $\bb(P_1,P_1)$ as $\bb(P_1)$. 

The article is organized as follows. In Section \ref{s2}, according to Lax pair, we give the direct scattering including Jost solutions, modified Jost solution and the reflection coefficient, and then perform inverse scattering transform to obtain the reconstructed formula by constructing the correspondent RH problem. 
In Section \ref{s3}, we analyze the long-time asymptotic behavior on the region $-1<-V_0\le \xi\deff\frac{n}{2t}\le V_9<1$.
In Section \ref{s4} and \ref{s5}, we analyze the long-time aaymptotics for $|\xi|\ge V_0>1$. 

\subsection{Main results}
\indent

In this paper, we study the long-time asymptotic analysis on three regions as shown in Figure \ref{f1}. 
\begin{figure}
\centering\includegraphics{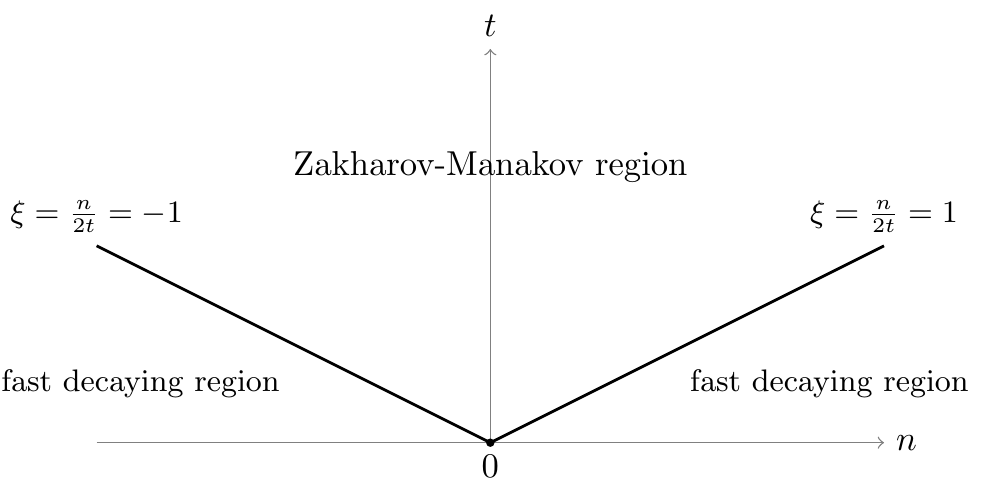}
\caption{Regions on the half plane divided by rays $\frac{n}{2t}=\pm1$.}\label{f1}
\end{figure}

When $-1<-V_0\le \xi\le V_0<1$, we see that first-order stationary phase points $S_j=i\xi+\sqrt{1-\xi^2}$ are on the jump contour $\Sigma$.
Since these stationary phase points appear on $\Sigma$, the reflection coefficient at $S_j$ has impact to the solution;
so, we come to investigate the local RH problem at $S_j$ on Section \ref{s3.4}, and make scaling transform for them to obtain the model RH problem that is related to parabolic cylinder functions, which refers to \cite{deift1993long}.
By the model RH problem, we obtain the oscillatory leading term that is $\oo(t^{-\frac{1}{2}})$ as shown in (\ref{e3s}), and we obtain that as $t\to+\infty$, the solution in Zakharov-Manakov type formulas (\ref{e3s}). 
It's notable that under the condition $\{q_n\}_{n=-\infty}^\infty\in L^{2,1}$, the decaying rate is $\oo(t^{-1})$. 

When $|\xi|\ge V_0>1$, stationary phase points are pure imaginary number and off the jump contour $\Sigma$;
in this case, the leading term in (\ref{e3s}) vanishes, and $q_n(t)$ decays fast when $t\to+\infty$,
and we call these regions as fast decaying region. 
\begin{theorem}\label{th1}
	For the initial-valued problem (\ref{e1})-(\ref{e2}), the solution $q_n(t)$ admits the following long-time asymptotic formula in the sectors obtained by dividing the half plane by rays $\frac{n}{2t}=\pm1$, which is shown in Figure \ref{f1}.
	\begin{enumerate}[label=\alph*)]
		\item The Zakharov-Manakov region $-1<-V_0\le \xi\le V_0<1$ with some positive constant $V_0$: as $t\to+\infty$, the solution admits Zakharov-Manakov type formulas:
		\begin{align}\label{e3s}
			&q_n(t)=\frac{i}{2}t^{-\frac{1}{2}}(1-\xi^2)^{-\frac{1}{4}}\delta^{-1}(0)\sum_{j=1}^{2}\delta_{j0}^{2}	[M^{L,j}_1]_{1,2}+\oo(t^{-\frac{3}{4}}),\\
			&\quad[M_1^{L,j}]_{1,2}=-i\frac{(2\pi)^\frac{1}{2}e^\frac{\pi i}{4}e^\frac{-\pi\nu_j}{2}}{r(S_j)\Gamma(-i\nu_j)},\quad j=1,2,\notag
		\end{align}
	where $r(S_j)$ is the value of reflection coefficients at stationary phase points $S_j$, $j=1,2$, $\int^{S_1}_{S_2}$ denotes the integral on $\Sigma$ from $S_2$ to $S_1$, $\Gamma(\lambda)$ is the Euler's gamma function, 
	\begin{align*}
		&S_j=-i\xi+(-1)^j\sqrt{1-\xi^2},\quad \beta_j=\frac{it^{-\frac{1}{2}}}{2}(1-\xi^2)^{-\frac{1}{4}}S_j,\quad\nu_j=-\frac{1}{2\pi}\ln(1-|r(S_j)|^2) \\
		&\delta^{-1}(0)=e^{\frac{1}{2\pi i}\int_{S_2}^{S_1}s^{-1}\ln(1-|r(s)|^2)\ddddd s},\quad\phi(S_j)=2((-1)^j\sqrt{1-\xi^2}-\xi\arg S_j-1),\\
		&\alpha_j(S_j)=\frac{1}{2\pi i}\int_{S_2}^{S_1}\frac{\ln(1-|r(s)|^2)-\ln(1-|r(S_j)|^2)}{s-S_j}\ddddd s,\\
		&\delta_{j0}=e^{\alpha_j(S_j)-\frac{it}{2}\phi(S_j)}\left(\frac{(-1)^{j-1}\beta_j}{S_1-S_2}\right)^{(-1)^{j-1}i\nu_j}.
	\end{align*}
\item The fast decaying region $|\xi|\ge V_0>1$: as $t\to+\infty$, the solution decays to zero
\begin{align*}
	q_n(t)\sim\oo(t^{-1}).
\end{align*}
	\end{enumerate}

\end{theorem}

\section{Riemann-Hilbert problem and the solution for defoxusing Ablowitz-Ladik systems}\label{s2}
\indent

In this section, we investigate Jost solutions: $X^\pm$ for the Lax pair corespondent to the initial-valued problem, make a transform for Jost solutions to obtain a modified Jost solution: $X^\pm\rightsquigarrow Y^\pm$, and then construct the RH problem. 
The Lax pair for (\ref{e1}) is:
\begin{subequations}
\begin{align}
	&X(z,n+1,t)=\left[\begin{matrix}
		z&q_n(t)\\\overline{q_n(t)}&z^{-1}
	\end{matrix}\right]X(z,n,t)\label{e3},\\
&\partial_t X(z,n,t)=i\left[\begin{matrix}
	-\frac{1}{2}(z-z^{-1})^2+q_n(t)\overline{q_{n-1}(t)}&q_{n-1}(t)z^{-1}-q_n(t)z\\\overline{q_n(t)}z^{-1}-\overline{q_{n-1}(t)}z&\frac{1}{2}(z-z^{-1})^2-q_{n-1}(t)\overline{q_{n}(t)}
\end{matrix}\right]\notag\\
&\quad\times X(z,n,t).
\end{align}
\end{subequations}
where $z$ is the spectral parameter and the Lax pair admits the Jost solution,
\begin{align}\label{e5}
	X^\pm\equiv X^\pm(z,n,t)\sim z^{n\sigma_3}e^{-\frac{it}{2}(z-z^{-1})^2\sigma_3},\quad n\to\pm\infty. 
\end{align}
Seeing from the spatial problem (\ref{e3}), we derive that
\begin{align}\label{e6}
	\det X^-=\prod_{k=-\infty}^{n-1}(1-|q_k(t)|^2),\quad\det X^+=\prod_{k=n}^{\infty}(1-|q_k(t)|^2)^{-1}.
\end{align}
Introducing
\begin{align*}
	c_n(t)=\prod_{k=n}^{\infty}(1-|q_k(t)|^2),\quad n\in\mathbb{Z}\cup\{-\infty\},
\end{align*}
because of assumption (\ref{e2}), from the result of \cite{chen2022sobolev}, we learn that $c_{-\infty}=\lim_{n\to-\infty}c_n(t)$ is nonzero and independent on $t$; therefore, for any fixed $t\le0$,
\begin{align*}
	1-|q_n(t)|^2\ne0,\quad n\in\mathbb{Z},
\end{align*}
which combined with (\ref{e3}) deduces that $X^\pm$ are invertible and fundamental solutions for the Lax pair; 
moreover, by the uniqueness of solution, there is a unique $2\times2$ matrix-valued function depending only on $z$ such that
\begin{align}\label{e7}
	X^-(z,n,t)=X^+(z,n,t)S(z),\quad S(z)=\left[\begin{matrix}
		a(z)&\breve b(z)\\b(z)&\breve a(z)
	\end{matrix}\right].
\end{align}
$S(z)$ is also known as the scattering matrix. 
Since $\left[\begin{matrix}
	0&q_n(t)\\\overline{q_n(t)}&0
\end{matrix}\right]$ is Hermitian, we learn from the spatial problem (\ref{e3}) that $X^\pm$ admit the following symmetric property
\begin{align*}
	X^\pm(z,n,t)=\sigma_1\overline{X^\pm(\bar z^{-1},n,t)}\sigma_1,\quad \sigma_1=\left[\begin{matrix}
		0&1\\1&0
	\end{matrix}\right],
\end{align*}
which combined with (\ref{e7}) deduces symmetry properties of $S(z)$:
\begin{align}\label{e9}
	\breve a(z)=\overline{a(\bar z^{-1})},\quad \breve b(z)=\overline{b(\bar z^{-1})},\quad z\in\mathbb{C}. 
\end{align}
Making the transformation
\begin{align}\label{e8}
	Y^\pm(z,n,t)=\left[\begin{matrix}
		1&0\\0&z
	\end{matrix}\right]X^\pm(z,n,t)z^{-n\sigma_3}e^{\frac{it}{2}(z-z^{-1})^2\sigma_3}\left[\begin{matrix}
	1&0\\0&z^{-1}
\end{matrix}\right],
\end{align}
from property (\ref{e5}), we learn that $Y^\pm(z,n,t)$ satisfy the normalization property
\begin{align}
	Y^\pm(z,n,t)\sim I,\quad n\to\pm\infty.
\end{align}
Making the spectral parameter transform: $\lambda=z^2$, seeing (\ref{e2}) and (\ref{e9}), we deduce that $Y^\pm(z,n,t)=Y^\pm(-z,n,t)$;
thus, we simply write $Y^\pm(\lambda,n,t)=Y^\pm(z,n,t)$. 
Seeing from \cite{chen2022sobolev}, we learn that $Y^+_1$ and $Y^-_2$ are both holomorphic on $D^-=\{|\lambda|<1\}$ and continuously extended to $D^-\cup\Sigma$, $Y^-_1$ and $Y^+_2$ are both holomorphic on $D^+=\{|\lambda|>1\}$ and continously extended to $D^+\cup\Sigma$. 
By (\ref{e7}) and (\ref{e8}), we derive that
\begin{align}\label{e11}
	a(z)=c_n(t)\det[Y^-_1,Y^+_2](\lambda,n,t),\quad zb(z)=c_n(t)e^{-it\phi(\lambda,n,t)}\det[Y^+_1,Y^-_1](\lambda,n,t),
\end{align}
where
\begin{align*}
	\phi(\lambda,n,t)=\lambda+\lambda^{-1}+2i\xi\log\lambda-2;
\end{align*}
since then, we write $a(\lambda=z^2)$ instead of $a(z)$.
Because $[Y^+_1,Y^-_2]$ is holomorphic on $D^-$ and continuously extended to $D^-\cup\Sigma$, by (\ref{e11}), $a(\lambda)$ is holomorphic on $D^-$ and continuously extended to $D^+\cup\Sigma$, 
%;by the symmetry (\ref{e9}), we also denote $\breve a(z)$ as $\breve a(\lambda)$ and it is holomorphic on $D^+$ and continuous on $\overline{D^+}$
too;
moreover, because of the continuity of $[Y^+_1,Y^-_1]$ on $\Sigma$, $zb(z)$ is also continuous on $\Sigma$;
therefore, we define the reflection coefficients on the circle $\lambda=e^{i\theta}\in\Sigma$:
\begin{align*}
	r(\lambda)=\frac{zb(z)}{a(z)},
\end{align*}
which is well-defined and belongs to $H^1_\theta(\Sigma)$ according to \cite{chen2022sobolev}.
Sometimes, we denote $r(\theta)=r(e^{i\theta})$ without confusion of notation on the jump contour $\Sigma$. 
Since $r\in H^1_\theta(\Sigma)$, by Sobolev embedding theory, we learn that $r(\theta)$ is $\frac{1}{2}$-H\"older continuous and boundend on $\Sigma$. 
By (\ref{e6}), (\ref{e7}) and (\ref{e9}), we see that
\begin{align}\label{e12s}
	1-|r(\theta)|^2=\frac{c_{-\infty}}{|a(e^{i\theta})|}>0,\quad \theta\in[0,2\pi]. 
\end{align}
Seeing from \cite{chen2022sobolev} and setting $Y^\pm=[Y^\pm_1,Y^\pm_2]$, we learn that as $\lambda\to\infty$ and $\lambda\to 0$, $Y^\pm(z,n,t)$ admit the following asymptotic property:
\begin{subequations}
\begin{align}
	&[Y^+_1,Y^-_2](\lambda,n,t)\sim \left[\begin{matrix}
		c_n^{-1}(t)+O(\lambda)&q_{n-1}(t)+O(\lambda)\\-c_n^{-1}(t)\overline{q_n(t)}\lambda+O(\lambda^2)&1+O(\lambda)
	\end{matrix}\right],\quad \lambda\to0,\label{e12}\\
&[Y^-_1,Y^+_2](\lambda,n,t)\sim\left[\begin{matrix}
	1+O(\lambda^{-1})&-c_n^{-1}(t)q_n(t)\lambda^{-1}+O(\lambda^{-2})\\\overline{q_{n-1}(t)}+O(\lambda^{-1})&c_n^{-1}(t)+O(\lambda^{-1})
\end{matrix}\right],\quad \lambda\to\infty\label{e13}.
\end{align}
\end{subequations}
From (\ref{e11}) and (\ref{e13}), we learn that
\begin{align}\label{e15s}
	\lim_{\lambda\to\infty}a(\lambda)=1.
\end{align}

We come to the construction of corespondent RH problems. 
Set a $2\times2$ matrix-valued function
\begin{align}\label{e14}
	M\equiv M(\lambda,n,t)=\begin{cases}
		\left[\begin{matrix}
			1&0\\-\overline{q_{n-1}(t)}c_n(t)&c_n(t)
		\end{matrix}\right]\left[Y_1^+,\frac{Y_2^-}{\breve a}\right](\lambda,n,t)&|\lambda|<1,\\
	\left[\begin{matrix}
		1&0\\-\overline{q_{n-1}(t)}c_n(t)&c_n(t)
	\end{matrix}\right]\left[\frac{Y_1^-}{a},Y_2^+\right](\lambda,n,t)&|\lambda|>1;
	\end{cases}
\end{align}
then, we clain that $M$ admits the following RH problem.
The first item comes from that of $Y^\pm(\lambda,n,t)$ and $a(\lambda)$. 
The second item naturally follows after the definition of $M(\lambda,n,t)$.
(\ref{e8}) and (\ref{e14}) deduce the third item. 
Here, the orientation of  jump contour is clockwise, and as general notation, we call the left/right side of the jump contour as $+/-$ side.
As follows, we obtain the reconstructed formula from (\ref{e9}), (\ref{e12}), (\ref{e15s}) and (\ref{e14}) that the reconstructed formula is
\begin{align}\label{e15}
	q_n(t)=M_{1,2}(0,n+1,t).
\end{align}
\begin{rhp}\label{r2.1}
	Find a $2\times2$ matrix-valued function $M$ such that
	\begin{itemize}
		\item $M$ is analytic on $\Sigma$.
		\item As $\lambda\to\infty$, $M\sim I+O(\lambda^{-1})$.
		\item On $\lambda\in\Sigma$,
		\begin{align*}
			M_+=M_-V,\quad V\equiv V(\lambda,n,t)=\left[\begin{matrix}
				1-|r(\lambda)|^2&-\overline{r(\lambda)}e^{-it\phi(\lambda,n,t)}\\r(\lambda)e^{it\phi(\lambda,n,t)}&1
			\end{matrix}\right]. 
		\end{align*}
	\end{itemize}
\end{rhp}

\section{Long-time asymptotic analysis on $-1<-V_0\le\xi\le V_0<1$}\label{s3}
\indent

In this section, we study the case for the region $-1<-V_0\le\xi\le V_0<1$. In this case, we derive that there are two first-order stationary phase points $S_1$ and $S_2$ on the unit circle $\Sigma$:
therefore, we obtain that
\begin{align}\label{e17s}
	&\phi(\lambda,n,t)-\phi(S_j,n,t)\sim\frac{\phi''(S_j,n,t)}{2}(\lambda-S_j)^2+\oo(\lambda-S_j)^3,\\
	 &\quad\phi''(S_j,n,t)=(-1)^{j}S_j^{-2}\sqrt{1-\xi^2}.\notag
\end{align}

\begin{figure}
	\centering\includegraphics[width=0.4\linewidth]{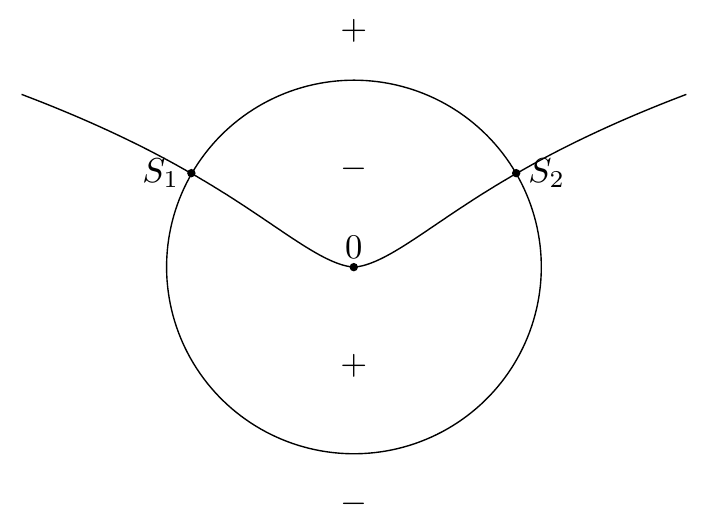}
	\caption{The sign of $\im\phi$ in region I; $S_1$, $S_2$ the stationary phase point}\label{f2}
\end{figure}

\subsection{Deformation on the jump contour}
\indent

In this part, we study an RH problem transform: $M\rightsquigarrow M^{(1)}=M\delta^{-\sigma_3}$;
and the new jump matrix admits a proper factorization, seeing in RH problem \ref{r3.2}. 

Introducing the scalar function
\begin{align}\label{e16}
	\delta\equiv\delta(\lambda)=e^{\frac{1}{2\pi i}\int_{S_2}^{S_1}\frac{\ln(1-|r(s)|^2)}{s-\lambda}\mathrm{d}s},
\end{align}
where the integral is along the lower-half arc on $\Sigma$: $\wideparen{S_2S_1}$ and this arc's orientation is from $S_2$ to $S_1$ along $\Sigma$, we derive that $\delta$ admits properties shown in Proposition \ref{p3.1}
\begin{proposition}\label{p3.1}
	Since $r\in H^1_\theta(\Sigma)$, the scalar function $\delta$ satisfies: 
	\begin{enumerate}[label=(\alph*)]
		\item 
		 $\delta$ is analytic on $\mathbb{C}\setminus \wideparen{S_2S_1}$.
		\item 
		As $\lambda\to\infty$, $\delta(\lambda)\sim 1+O(\lambda^{-1})$.
		\item On the arc $\lambda\in\wideparen{S_2S_1}$. $\delta_+(\lambda)=\delta_-(\lambda)(1-|r(\lambda|^2))$.
		\item On $\lambda\in\mathbb{C}\setminus\wideparen{S_2S_1}$, $\delta$ admits the symmetry:
		\begin{align*}
			\delta(\lambda)=\overline{\delta(0)/\delta(\bar\lambda^{-1})}.
			\end{align*}
		\item On the neighborhood of $S_1$ and $S_2$, $\delta$ admits following asymptotic properties:
		\begin{align*}
			&\delta(\lambda)\sim\left(\frac{\lambda-S_1}{\lambda-S_2}\right)^{i\nu_j}e^{\alpha_j(S_j)}+O(|\lambda-S_j|^{\frac{1}{2}}),\quad \lambda\to S_j,\\
			&\quad \alpha_j(\lambda)=\frac{1}{2\pi i}\int_{S_2}^{S_1}\frac{\ln(1-|r(s)|^2)-\ln(1-|r(S_j)|^2)}{s-\lambda}\ddddd s,\\
			&\quad \nu_j=-\frac{\ln(1-|r(S_j)|^2)}{2\pi},\quad\quad j=1,2. 
		\end{align*}
		
	\end{enumerate}
\end{proposition}
\begin{proof}
	(a) comes from the analyticity of $Y^\pm(\lambda,n,t)$, $a(\lambda)$ and (\ref{e14}). 
	(b) comes from (\ref{e13}) and (\ref{e14}). 
	(c) comes from (\ref{e7}) and (\ref{e8}). 
	Since $\delta(\lambda)$ satisfies (a-c) which determine an RH problem with the jump contour $\wideparen{S_2S_1}$, we also verify by computation that $\overline{\delta(0)/\delta(\bar\lambda^{-1})}$ is also the solution of this RH problem; 
	then, by uniqueness of solution, we derive (d). 
	Changing the variable on $\in\Sigma$: $\lambda\to\lambda'$, $s\to s'$,
	\begin{align*}
		\lambda=\sqrt{S_1S_2}\frac{\sqrt{S_1}+\sqrt{S_2}\lambda'}{\sqrt{S_2}+\sqrt{S_1}\lambda'},\quad s=\sqrt{S_1S_2}\frac{\sqrt{S_1}+\sqrt{S_2}s'}{\sqrt{S_2}+\sqrt{S_1}s'},
	\end{align*}
	we obtain that
	\begin{align}\label{e17}
		\int_{S_2}^{S_1}\frac{f(s)-f(S_1)}{(2\pi i)(s-\lambda)}\mathrm{d}s=\frac{\sqrt{S_2}+\sqrt{S_1}\lambda'}{2\pi i}\int_{-\infty}^{0}\frac{f(s)-f(S_1)\ddddd s'}{(\sqrt{S_2}+\sqrt{S_1}s')(s'-\lambda')}, 
	\end{align}
	where $f(s)=\ln(1-|r(s)|^2)$.
	Since $r(s)\in H^1_\theta(\Sigma)\subset L^\infty(\Sigma)$, it naturally follows that
	\begin{align}\label{e18}
		F(s')=\begin{cases}
			\frac{f(s)-f(S_1)}{\sqrt{S_2}+\sqrt{S_1}s'},&s'\le0,\\
			0,&s'>0,
		\end{cases}
	\end{align}
	belongs to $H^1(\mathbb{R})$ and 
	\begin{align*}
		F(s=0)=0;
	\end{align*}
	therefore, we apply Lemma 23.3 in \cite{beals1988direct} and get that for any $\lambda\in\mathbb{C}\setminus\mathbb{R}$
	\begin{align}
		&\Big|\int_{-\infty}^{+\infty}\frac{F(s')}{(2\pi i)(s'-\lambda')}\mathrm{d}s'-\int_{-\infty}^{+\infty}\frac{F(s')}{2\pi is'}\mathrm{d}s'\Big|\lesssim\parallel F\parallel_{H^1}|\lambda'|^{\frac{1}{2}}\label{e19}
	\end{align}
	Considering (\ref{e17}), (\ref{e18}) and (\ref{e19}), we obtain (e) for $j=1$, and the proof for $j=2$ is parallel.
	We confirm the result.
\end{proof}
\begin{remark}
	Seeing property (c) and (d) in Proposition \ref{p3.1}, we learn that for $\lambda\in\wideparen{S_2S_1}$, 
	\begin{align*}
		|\delta_+(\lambda)|^2=(1-|r(s)|^2)\delta(0),\quad |\delta_-(\lambda)|^2=\delta(0)(1-|r(\lambda)|^2)^{-1}.
	\end{align*}
	which are nonzero and finite by (\ref{e12s}) and (\ref{e17});
	moreover, $\delta(\lambda)\to1$ as $\lambda\to\infty$;
	so, $\delta$ and $\delta^{-1}$ are both bounded function by (a) in Proposition \ref{p3.1} and the maximal modular theorem. 
	
\end{remark}

Here, we confirm that a new $2\times2$ matrix-valued function
\begin{align}\label{e20}
	M^{(1)}=M\delta^{-\sigma_3}, 
\end{align}
satisfies RH problem \ref{r3.2}, which is the natural consequence of RH problem \ref{r2.1}, (a-c) in Proposition \ref{p3.1} and (\ref{e20}). 
\begin{rhp}\label{r3.2}
	Find a $2\times2$ matrix-valued function such that
	\begin{itemize}
		\item $M^{(1)}$ is analytic on $\mathbb{C}\setminus\Sigma$.
		\item As $\lambda\to\infty$, $M^{(1)}(\lambda,n,t)\sim I+O(\lambda^{-1})$.
		\item On $\lambda\in\Sigma$,
		\begin{align*}
			M^{(1)}_+=M^{(1)}_-V^{(1)},\quad 
		\end{align*}
		where $V^{(1)}\equiv V^{(1)}(\lambda,n,t)$ is the jump matrix
		\begin{align*}
			&V^{(1)}=\begin{cases}
				\left[\begin{matrix}
					1&-\delta_+^2\frac{\overline{r(\lambda)}}{1-|r(\lambda)|^2}e^{-it\phi(\lambda,n,t)}\\
					\delta_-^{-2}\frac{r(\lambda)}{1-|r(\lambda)|^2}e^{it\phi(\lambda,n,t)}&1-|r(\lambda)|^2
				\end{matrix}\right]\\
		\left[\begin{matrix}
			1-|r(\lambda)|^2&-\overline{r(\lambda)}\delta^2(\lambda)e^{-it\phi(\lambda,n,t)}\\r(\lambda)\delta^{-2}(\lambda)e^{it\phi(\lambda,n,t)}&1
		\end{matrix}\right]
	\end{cases}\\
&\quad =\begin{cases}
	\left[\begin{matrix}
		1&0\\
		\frac{\delta_-^{-2}r(\lambda)}{1-|r(\lambda)|^2}e^{it\phi(\lambda,n,t)}&1
	\end{matrix}\right]\left[\begin{matrix}
	1&-\frac{\delta_+^2\overline{r(\lambda)}}{1-|r(\lambda)|^2}e^{-it\phi(\lambda,n,t)}\\
	0&1
\end{matrix}\right]&\lambda\in\wideparen{S_2S_1},\\
	\left[\begin{matrix}
		1&-\overline{r(\lambda)}\delta^2(\lambda)e^{-it\phi(\lambda,n,t)}\\0&1
	\end{matrix}\right]\left[\begin{matrix}
	1&0\\r(\lambda)\delta^{-2}(\lambda)e^{it\phi(\lambda,n,t)}&1
\end{matrix}\right]&\lambda\in\Sigma\setminus\wideparen{S_2S_1}.
\end{cases}
		\end{align*}
	\end{itemize}
\end{rhp}

\subsection{Split the circle}\label{s3.2}
\indent

In this part, we introduce a new RH problem transform: $M^{(1)}\rightsquigarrow M^{(2)}$.
As a result, the jump contour $\Sigma$ is deformed into $\Sigma^{(2)}$, where the jump matrix is decaying.

Introducing the jump contour $\Sigma^{(2)}=L\cup\tilde L$ as shown in Figure \ref{f3}, we see that $L$ and $\tilde L$ are both closed curves consisting of straight line segments and arcs centering at the origin. 
As for orientation of the jump contour, we denote it in  Figure \ref{f3}. 
Again seeing Figure \ref{f3}, the complex plane is divided into six parts by $\Sigma^{(2)}$ and $\Sigma$: $\Omega_j$, $j=1,\dots,6$. 
\begin{figure}
	\centering\includegraphics[width=0.5\linewidth]{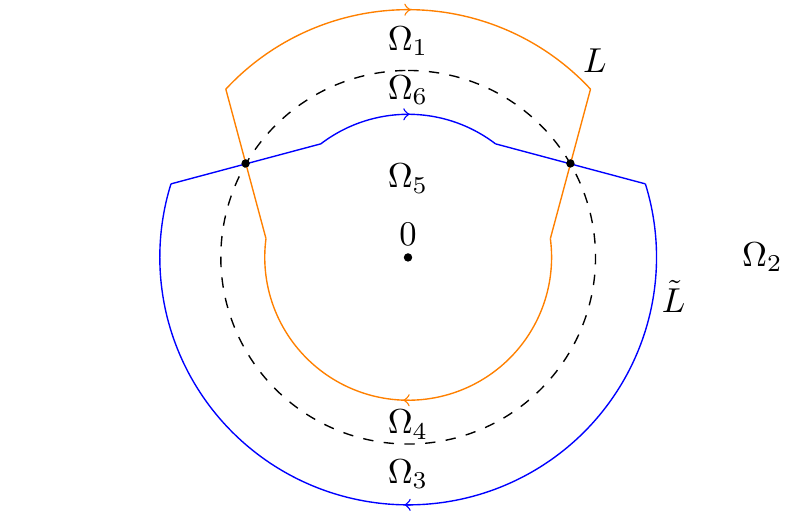}
	\caption{Jump contour $\Sigma^{(2)}=L\cup \tilde L$.}\label{f3}
\end{figure}
Since we have constructed the jump contour, introduce a $2\times2$ matrix-valued function
\begin{align}\label{e23}
	M^{(2)}=M^{(1)}\pp,
\end{align}
where the $2\times2$ matrix-valued function $\pp$ is invertible and written as
\begin{align}\label{e24}
	\pp\equiv\pp(\lambda,n,t)=\begin{cases}\left[\begin{matrix}
			1&0\\-R_1(\lambda)\delta^{-2}(\lambda)e^{it\phi(\lambda,n,t)}&1
		\end{matrix}\right]&\lambda\in\Omega_1,\\
		\left[\begin{matrix}
			1&R_3(\lambda)\delta^2(\lambda)e^{-it\phi(\lambda,n,t)}\\0&1
		\end{matrix}\right]&\lambda\in\Omega_3,\\
		\left[\begin{matrix}
			1&0\\R_4(\lambda)\delta^{-2}(\lambda)e^{it\phi(\lambda,n,t)}&1
		\end{matrix}\right]&\lambda\in\Omega_4,\\
		\left[\begin{matrix}
			1&-R_6(\lambda)\delta^{2}(\lambda)e^{-it\phi(\lambda,n,t)}\\0&1
		\end{matrix}\right]&\lambda\in\Omega_6,\\
		I&\lambda\in\Omega_2\cup\Omega_5
	\end{cases}
\end{align}
and we denote $\lambda=\rho e^{i\theta}$, $\rho\ge0$, $\theta\in[0,2\pi]$,  and
\begin{align}\label{e25s}
	R_1(\lambda)=r(\theta),\quad R_3(\lambda)=\frac{\overline{r(\theta)}}{1-|r(\theta)|^2},\quad R_4(\lambda)=\frac{r(\theta)}{1-|r(\theta)|^2},\quad R_6(\lambda)=\overline{r(\theta)}.
\end{align}
Recalling (\ref{e12s}) and that the reflection coefficients $r$ is bounded on $\Sigma$, we see by (\ref{e25s}) that $R_1(\lambda)$, $R_3(\lambda)$, $R_4(\lambda)$ and $R_6(\lambda)$ are bounded on $\mathbb{C}$. 
Then, as a consequence of RH problem \ref{r3.2} and (\ref{e23}), we deduce that $M^{(2)}$ satisfies $\bar\partial$-RH problem \ref{r3.3}, and the jump matrix  admits the lower/upper triangular factorization:
\begin{align*}
	V^{(2)}=(1-w_-)^{-1}(1+w_+),
\end{align*}
where we denote
\begin{align}
	&w=w_++w_-,\\
	&w_+=\begin{cases}
		\textbf{0}&L,\\
		\left[\begin{matrix}
			0&-\delta^2(\lambda)R_3(\lambda)e^{-it\phi(\lambda,n,t)}\\0&0
		\end{matrix}\right]&\tilde L\cap D_+,\\
		\left[\begin{matrix}
			0&-\delta^2(\lambda)R_6(\lambda)e^{-it\phi(\lambda,n,t)}\\0&0
		\end{matrix}\right]&\tilde L\cap D_-,
	\end{cases}\label{e27}\\
	&w_-=\begin{cases}
		\left[\begin{matrix}
			0&0\\\delta^{-2}(\lambda)R_1(\lambda)e^{it\phi(\lambda,n,t)}&0
		\end{matrix}\right]&L\cap D_+,\\
		\left[\begin{matrix}
			0&0\\\delta^{-2}(\lambda)R_4(\lambda)e^{it\phi(\lambda,n,t)}&0
		\end{matrix}\right]&L\cap D_-,\\
		\textbf{0}&\tilde L.\label{e28}
	\end{cases}
\end{align}
and it is easy to check that $w_\pm$ are $2\times2$ nilpotent matrices. 
\begin{drhp}\label{r3.3}
	Find a $2\times2$ matrix-valued function $M^{(2)}$ such that
	\begin{itemize}
		\item $M^{(2)}$ belongs to $C^0(\mathbb{C}\setminus(L\cup\tilde L))$ and its first-order partial derivatives are continuous on $\mathbb{C}\setminus(\Sigma\cup L\cup\tilde L)$.
		\item As $\lambda\to\infty$, $M^{(2)}\sim I+\oo(\lambda^{-1})$.
		\item On $\lambda\in\Sigma$, 
		\begin{align*}
			M^{(2)}_+=M^{(2)}_-V^{(2)},
		\end{align*}
		where
		\begin{align*}
			V^{(2)}\equiv V^{(2)}(\lambda,n,t)=\begin{cases}
				\left[\begin{matrix}
					1&0\\R_1(\lambda)\delta^{-2}(\lambda)e^{it\phi(\lambda,n,t)}&1
				\end{matrix}\right]&\lambda\in L\cap D_+,\\
			\left[\begin{matrix}
			1&-R_3(\lambda)\delta^2(\lambda)e^{-it\phi(\lambda,n,t)}\\0&1
		\end{matrix}\right]&\lambda\in \tilde L\cap D_+,\\
	\left[\begin{matrix}
	1&0\\R_4(\lambda)\delta^{-2}(\lambda)e^{it\phi(\lambda,n,t)}&1
\end{matrix}\right]&\lambda\in L\cap D_-,\\
\left[\begin{matrix}
1&-R_6(\lambda)\delta^{2}(\lambda)e^{-it\phi(\lambda,n,t)}\\0&1
\end{matrix}\right]&\lambda\in\tilde L\cap D_-.
			\end{cases}
		\end{align*}
		\item On $\mathbb{C}\setminus(\Sigma\cup L\cup\tilde L)$,
		\begin{align*}
			\bar\partial M^{(2)}=M^{(2)}\bar\partial \pp,
		\end{align*}
		where $\bar\partial\pp$ is a nilpotent matrix and
		\begin{align*}
			\bar\partial\pp(\lambda,n,t)=\begin{cases}\left[\begin{matrix}
					0&0\\-\bar\partial R_1(\lambda)\delta^{-2}(\lambda)e^{it\phi(\lambda,n,t)}&0
				\end{matrix}\right]&\lambda\in\Omega_1,\\
				\left[\begin{matrix}
					0&\bar\partial R_3(\lambda)\delta^2(\lambda)e^{-it\phi(\lambda,n,t)}\\0&0
				\end{matrix}\right]&\lambda\in\Omega_3,\\
				\left[\begin{matrix}
					0&0\\\bar\partial R_4(\lambda)\delta^{-2}(\lambda)e^{it\phi(\lambda,n,t)}&0
				\end{matrix}\right]&\lambda\in\Omega_4,\\
				\left[\begin{matrix}
					0&-\bar\partial R_6(\lambda)\delta^{2}(\lambda)e^{-it\phi(\lambda,n,t)}\\0&0
				\end{matrix}\right]&\lambda\in\Omega_6,\\
				\textbf{0}&\lambda\in\Omega_2\cup\Omega_5.
			\end{cases}
		\end{align*}
	\end{itemize}
\end{drhp}

\subsection{Deformation of the $\bar\partial$-RH problem}\label{s3.3}
\indent

In this part, we deform the solution for $\bar\partial$-RH problem \ref{r3.3} into the product of solutions for RH problem \ref{r3.4} and $\bar\partial$-problem \ref{d3.6}: 
\begin{align}\label{e22}
	M^{(2)}=M^{(2,D)}M^{(2,R)},
\end{align}
where $M^{(2,R)}$ admits the same jump condition as $M^{(2)}$'s, and $M^{(2,D)}$ is continuous over complex plane $\mathbb{C}$. 
Seeing from RH problem \ref{r3.4}, $M^{(2,R)}$ has no pole and on $\lambda\in L\cup\tilde L$,
\begin{align*}
	\det M^{(2,R)}_+(\lambda,n,t)=\det M^{(2,R)}_-(\lambda,n,t);
\end{align*}
so, it is analytic in the whole complex plane; moreover, we see that as $\lambda\to \infty$,
\begin{align*}
	\det M^{(2,R)}(\lambda,n,t)\sim 1+\oo(\lambda^{-1}),
\end{align*}
which deduce by Liouville's Theorem that
\begin{align*}
	\det M^{(2,R)}\equiv1,
\end{align*}
and $M^{(2,R)}$ is invertible. 
Since $M^{(2,R)}$ is invertible, we obtain that (\ref{e22}) is well-defined. 
\begin{rhp}\label{r3.4}
	Find a $2\times2$ matrix-valued function $M^{(2,R)}$ such that:
	\begin{itemize}
		\item $M^{(2,R)}$ is analytic on $\mathbb{C}\setminus L\cap\tilde L$.
		\item As $\lambda\to\infty$, $M^{(2,R)}\sim I+\oo(\lambda^{-1})$.
		\item On $\lambda\in L\cap\tilde L$, $M^{(2,R)}_+=M^{(2,R)}_-V^{(2)}$.
		
	\end{itemize}
\end{rhp}

\begin{dbarproblem}\label{d3.6}
	Find a $2\times2$ matrix-valued function $M^{(2,D)}$ such that
	\begin{itemize}
		\item $M^{(2,D)}$ belongs to $C^0(\mathbb{C})$ and its first-order partial derivatives are continuous on $\mathbb{C}\setminus(\Sigma\cap L\cap\tilde L)$.
		\item As $\lambda\to\infty$, $M^{(2,D)}\sim I+\oo(\lambda^{-1})$. 
		\item On $\lambda\in L\cap\tilde L$, $\bar\partial M^{(2,D)}=M^{(2,D)}\tilde\pp$, where $ \tilde\pp=M^{(2,R)}\bar\partial\pp (M^{(2,R)})^{-1}$.
	\end{itemize}
\end{dbarproblem}

Introducing a Cauchy-like operator $C_w$, for a $2\times2$ matrix-valued function $f$:
\begin{align*}
	&C_wf=C_+(fw_-)+C_-(fw_+), \quad w=w_++w_-,\\
	&C_\pm f(\lambda)=\lim_{\lambda'\to\lambda, \lambda' \text{ on $\pm$ side of $\Sigma^{(2)}$} }\int_{\Sigma^{(2)}}\frac{f(s)\ddddd s}{s-\lambda'},
\end{align*}
for RH problem \ref{r3.2}, since $(1-C_w)^{-1}\in\bb(L^2(\Sigma^{(2)}))$, which is shown in Section \ref{s3.9}, we obtain Beals-Coifman solutions:
\begin{align}\label{e25}
	M^{(2,R)}(\lambda,n,t)=I+\frac{1}{2\pi i}\int_{\Sigma}\frac{((1-C_w)^{-1}Iw)(s,n,t)\ddddd s}{s-\lambda}.
\end{align}

\subsection{RH problems at stationary phase points}\label{s3.4}
\indent

In this part, we introduce two crosses contained in $\Sigma^{(2)}$ associated to the two stationary phase points.
And then, by these crosses, we analyze the solution of RH problem \ref{r3.4} at $\lambda=0$. 
We find that as $t\to+\infty$, the leading terms is related to RH problems on each of the crosses, and the remaining part decays not faster than $\oo(t^{-1})$, i.e. the result in (\ref{e60}). 

\begin{figure}
	\centering\includegraphics{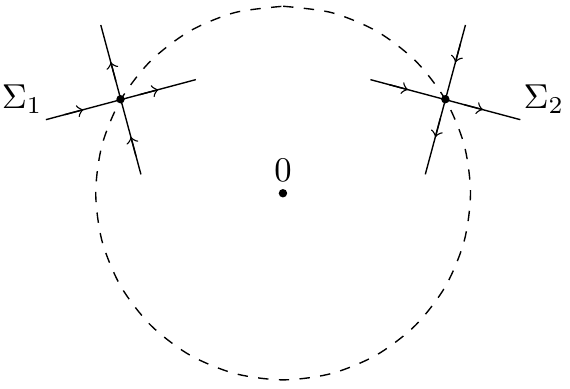}
	\caption{The jump contour $\Sigma'=\Sigma_1\cup\Sigma_2$.}
\end{figure}

Set $\epsilon_0>0$ fixed and introduce contours
\begin{align*}
	&\Sigma_1=\cup_{l=1}^4\Sigma_{1l},\quad\Sigma_{1l}=\left\{S_1(1+xe^{\frac{2l-3}{4}\pi i}):x\in[0,\epsilon_0)\right\},\\
	&\Sigma_2=\cup_{l=1}^4\Sigma_{2l},\quad\Sigma_{2l}=\left\{S_2(1+xe^{\frac{2l-3}{4}\pi i}):x\in[0,\epsilon_0)\right\}.
\end{align*}
We choose $\epsilon_0$ sufficiently small such that $\Sigma'\deff\Sigma_1\cup\Sigma_2\subset\Sigma^{(2)}$, and on $\lambda\in\Sigma'$,  (\ref{e17s}) guarantees that
\begin{subequations}\label{e31s}
	\begin{align}
		&\im\phi(\lambda,n,t)\ge x^2\sqrt{1-\xi^2},\quad \lambda=S_j(1+xe^{(-1)^{j}\frac{\pi i}{4}}),\quad x\in(-\epsilon_0,\epsilon_0),\\
		&\im\phi(\lambda,n,t)\le-x^2\sqrt{1-\xi^2},\quad \lambda=S_j(1+xe^{(-1)^{j-1}\frac{\pi i}{4}}),\quad x\in(-\epsilon_0,\epsilon_0).
	\end{align}
\end{subequations}
We also reasonably restrict $\epsilon_0$ such that for $\lambda\in\Sigma'$,
\begin{align}\label{e30}
	|\re\lambda|\ge\frac{1}{2}|\re S_1|=\frac{\sqrt{1-\xi^2}}{2}\ge\frac{\sqrt{1-V_0^2}}{2}.
\end{align}
Set $2\times2$ matrix-valued functions supported on $\Sigma'$:
\begin{align}\label{e31}
	&w'=w_-'+w_+',\quad w_\pm'=\begin{cases}
		w_\pm&\lambda\in\Sigma',\\
		0&\lambda\in\Sigma^{(2)}\setminus\Sigma'.
	\end{cases}
\end{align}
Similar to $C_w$, we also define a Cauchy-like operator $C_{w'}$, and in Section \ref{s3.9}, we show that $(1-C_{w'})^{-1}\in\bb(L^2(\Sigma^{(2)}))$ for sufficiently large $t$. 
Taking $\lambda=0$ in (\ref{e25}), since $1-C_{w}$ and $1-C_{w'}$ is invertible on $L^2(\Sigma^{(2)})$, we have by second resolvent identity that
\begin{align}\label{e32}
	&M^{(2,R)}(0,n,t)=I+\frac{1}{2\pi i}\int_{\Sigma^{(2)}}\frac{((1-C_w)^{-1}Iw)(s,n,t)\ddddd s}{s}\notag\\
	&\quad =I+\frac{1}{2\pi i}\int_{\Sigma^{(2)}}\frac{((1-C_{w'})^{-1}Iw')(s,n,t)\ddddd s}{s}+I_1+I_2,
\end{align}
where
\begin{align*}
	&I_1=\frac{1}{2\pi i}\int_{\Sigma^{(2)}}s^{-1}[(1-C_w)^{-1}I(w-w')](s,n,t)\ddddd s,\\
	&I_2=\frac{1}{2\pi i}\int_{\Sigma^{(2)}}s^{-1}[(1-C_w)^{-1}C_{w-w'}(1-C_{w'})^{-1}Iw'](s,n,t)\ddddd s.
\end{align*}

\begin{lemma}\label{l3.7}
	Since $r\in H^1_\theta(\Sigma)$, $w_\pm$ and $w$ are bounded functions on $\Sigma^{(2)}$. 
	Moreover, setting $w^r=w-w'$ and $w_\pm^r=w_\pm-w_\pm'$ on $\Sigma^{(2)}$, we obtain that there is a positive constant $C>0$ such that for $\lambda\in\Sigma^{(2)}$:
	\begin{align*}
		|w_\pm^r(\lambda)|\lesssim e^{-Ct},\quad |w^r(\lambda)|\lesssim e^{-Ct}. 
	\end{align*}
\end{lemma}
\begin{proof}
	Recalling that $\delta$, $\delta^{-1}$, $R_j$ are bounded functions on $\Sigma^{(2)}$, by (\ref{e27}) and (\ref{e28}), we only have to check the boundedness of $\im\phi$; 
	moreover, seeing Figure \ref{f2}, since $\im\phi>0$ on $L$ and $\im\phi<0$ on $\tilde L$, we obtain that $w_\pm$ are bounded on $\Sigma^{(2)}$. 
	As for the estimates of $w^r_\pm$, seeing Figure \ref{f2}, $L\setminus \Sigma'$ is a compact on $\{\im\phi>0\}$, then there is a constant $C>0$ such that
	\begin{align}
		\im\phi\big|_{L\setminus\Sigma'}\ge C;
	\end{align}
	similarly, choosing proper $C$, we also have
	\begin{align}
		\im\phi\big|_{\tilde L\setminus\Sigma'}\le -C;
	\end{align}
	therefore, we confirm results for $w^r_\pm$ and $w^r$. 
\end{proof}

The next task is to prove that as $t\to+\infty$
\begin{align}\label{e35s}
	I_1,I_2\sim \oo(t^{-1}), 
\end{align}
By Schwartz's inequality, the boundedness of $(1-C_w)^{-1}$ and $(1-C_{w'})^{-1}$ shown in Section \ref{s3.9}, the fact that $\Sigma^{(2)}$ is a compact set contained in $\mathbb{C}\setminus\{0\}$ and Lemma \ref{l3.7}, we obtain that as $t\to+\infty$
\begin{align}
	|I_1|\lesssim \parallel w^r\parallel_{L^2(\Sigma^{(2)})}\lesssim\parallel w^r\parallel_{L^\infty(\Sigma^{(2)})}\sim\oo(t^{-1}).
\end{align} 
Similarly for $|I_2|$, also by Schwartz's inequality, the result in Section \ref{s3.9} and Lemma \ref{l3.7}, we have that as $t\to+\infty$,
\begin{align}\label{e35}
	&|I_2|\lesssim\parallel C_{w^r}\parallel_{L^2\to L^2}\parallel I\parallel_{L^2(\Sigma^{(2)})}\parallel w'\parallel_{L^2(\Sigma^{(2)})}\notag\\
	&\quad\lesssim\parallel C_{w^r}\parallel_{L^2\to L^2}\lesssim\parallel w^r\parallel_{L^\infty(\Sigma^{(2)})}\sim \oo(t^{-1}).
\end{align} 
In (\ref{e35}), the second inequality is confirmed by the fact that $I$ and $w'$ are bounded on $\Sigma^{(2)}$, where the boundedness of $w'$ is supported by Lemma \ref{l3.7} and (\ref{e31}); 
the third inequality is also correct because for a $2\times2$ matrix function $f$ on $\Sigma^{(2)}$,
\begin{align}\label{e40}
	&\parallel C_{w^r}f\parallel_{L^2(\Sigma^{(2)})}=\parallel C_+(fw_-^r)+C_-(fw_+^r)\parallel_{L^2(\Sigma^{(2)})}\notag\\
	&\quad\lesssim \parallel fw_+\parallel_{L^2(\Sigma^{(2)})}+\parallel fw_-\parallel_{L^2(\Sigma^{(2)})}\lesssim \parallel f\parallel_{L^2(\Sigma^{(2)})}\parallel w^r\parallel_{L^\infty(\Sigma^{(2)})},
\end{align}
by the fact that Cauchy integral operators $C_\pm$ are bounded on the $L^2$ space.
Seeing (\ref{e32}) and (\ref{e35s}), we derive that as $t\to+\infty$,
\begin{align}\label{e39}
	M^{(2,R)}(0,n,t)&=I+\frac{1}{2\pi i}\int_{\Sigma’}\frac{((1-C_{w'})^{-1}Iw')(s,n,t)\ddddd s}{s}+\oo(t^{-1}).
\end{align}

The remaining work is to separate the contribution of each cross in $\Sigma'$. 
Introducing
\begin{align*}
	&w^j=w^j_-+w^j_+\quad j=1,2,\\
	&w^j_\pm=(\lambda,n,t)\begin{cases}
		w_\pm(\lambda,n,t)&\lambda\in\Sigma_j,\\
		0&\lambda\in\Sigma^{(2)}\setminus\Sigma_j,
	\end{cases}
\end{align*}
we deduce that
\begin{align*}
	w_\pm'=w^1_\pm+w^2_\pm. 
\end{align*}
and also define correspondent Cauchy-like operators $C_{w^j}$, $j=1,2$. 
By (3.17) in \cite{deift1993steepest}, we obtain 
\begin{align}\label{e43}
	&(1-C_{w'})(1+C_{w^1}(1-C_{w^1})^{-1}+C_{w^2}(1-C_{w^2})^{-1})\notag\\
	&\quad=(1-C_{w^2}C_{w^1}(1-C_{w^1})^{-1}-C_{w^1}C_{w^2}(1-C_{w^2})^{-1}).
\end{align}
Considering (\ref{e39}) and (\ref{e43}), we obtain that
\begin{align}\label{e44}
	&M^{(2,R)}(0,n,t)=I+\sum_{j=1,2}\frac{1}{2\pi i}\int_{\Sigma_j}s^{-1}[(1-C_{w^j})^{-1}Iw^j](s,n,t)\ddddd s\notag\\
	&\quad +I_3+I_4+I_5+\oo(t^{-1}). 
\end{align}
where
\begin{subequations}\label{e45}
	\begin{align}
		I_3=\frac{1}{2\pi i}\int_{\Sigma'}&[s^{-1}(1+C_{w^1}(1-C_{w^1})^{-1}+C_{w^2}(1-C_{w^2})^{-1})\notag\\
		&(1-C_{w^2}C_{w^1}(1-C_{w^1})^{-1}-C_{w^1}C_{w^2}(1-C_{w^2})^{-1})^{-1}\notag\\
		&(C_{w^2}C_{w^1}(1-C_{w^1})^{-1}+C_{w^1}C_{w^2}(1-C_{w^2})^{-1})Iw'](s,n,t)\ddddd s,\label{e45a}\\
		I_4=\frac{1}{2\pi i}\int_{\Sigma'}&s^{-1}[(1-C_{w^2})^{-1}C_{w^2}Iw^1](s,n,t)\ddddd s,\label{e45b}\\
		I_5=\frac{1}{2\pi i}\int_{\Sigma'}&s^{-1}[(1-C_{w^1})^{-1}C_{w^1}Iw^2](s,n,t)\ddddd s.\label{e45c}
	\end{align}
\end{subequations}
For $I_3$, considering $(1-C_{w^j})^{-1}\in\bb(L^2(\Sigma^{(2)}))$ for $j=1,2$ in Section \ref{s3.9}, by Lamma \ref{l3.10},
\begin{align*}
	(1-C_{w^2}C_{w^1}(1-C_{w^1})^{-1}-C_{w^1}C_{w^2}(1-C_{w^2})^{-1})^{-1}\in\bb(L^2(\Sigma^{(2)}));
\end{align*}
also, applying the technique used in (\ref{e40}) on $C_{w^j}$, we claim that $C_{w^j}$ belongs to $\bb(L^2(\Sigma^{(2)}))$, $j=1,2$; thus, we obtain by (\ref{e45a}) and Schwartz inequality that
\begin{align}\label{e49}
	|I_3|&\lesssim\parallel(C_{w^2}C_{w^1}(1-C_{w^1})^{-1}+C_{w^1}C_{w^2}(1-C_{w^2})^{-1})I\parallel_{L^2(\Sigma^{(2)})}\parallel w'\parallel_{L^2(\Sigma^{(2)})},\notag\\
	&\le(\parallel C_{w^2}C_{w^1}(1-C_{w^1})^{-1}I\parallel_{L^2(\Sigma^{(2)})}+\parallel C_{w^1}C_{w^2}(1-C_{w^2})^{-1}I\parallel_{L^2(\Sigma^{(2)})})\notag\\
	&\quad\times\parallel w'\parallel_{L^2(\Sigma^{(2)})}
\end{align}
By computation, applying (\ref{e41a}) in Lemma \ref{l3.10}, Lemma \ref{l3.9},  $(1-C_{w^1})^{-1}\in\bb(L^2(\Sigma^{(2)}))$ shown in Section \ref{s3.9} and the fact that
\begin{align}\label{e48}
	&\parallel C_{w^1}I\parallel_{L^2(\Sigma^{(2)})}\le\parallel C_+(w^1_-)\parallel_{L^2(\Sigma^{(2)})}+\parallel C_-(w^1_+)\parallel_{L^2(\Sigma^{(2)})}\notag\\
	&\quad\le\parallel w^1_-\parallel_{L^2(\Sigma^{(2)})}+\parallel w^1_+\parallel_{L^2(\Sigma^{(2)})}\lesssim t^{-\frac{1}{4}},
\end{align}
we obtain estimates:
\begin{align}\label{e50}
	&\parallel C_{w^2}C_{w^1}(1-C_{w^1})^{-1}I\parallel_{L^2(\Sigma^{(2)})}\le\parallel C_{w^2}C_{w^1}(1-C_{w^1})^{-1}C_{w^1}I\parallel_{L^2(\Sigma^{(2)})}\notag\\
	&\quad+\parallel C_{w^2}C_{w^1}I\parallel_{L^2(\Sigma^{(2)})}\le\parallel C_{w^2}C_{w^1}\parallel_{L^2\to L^2}\parallel (1-C_{w^1})^{-1}\parallel_{L^2\to L^2}\parallel C_{w^1}I\parallel_{L^2(\Sigma^{(2)})}\notag\\
	&\quad+\parallel C_{w^2}C_{w^1}\parallel_{L^\infty\to L^2}\parallel I\parallel_{L^\infty(\Sigma^{(2)})}
	\sim\oo(t^{-\frac{3}{4}}).
\end{align}
Similarly, we obtain that
\begin{align}\label{e51s}
	\parallel C_{w^2}C_{w^1}(1-C_{w^1})^{-1}I\parallel_{L^2(\Sigma^{(2)})}\sim\oo(t^{-\frac{3}{4}}).
\end{align}
By (\ref{e49}), (\ref{e50}), (\ref{e51s}) and Lemma \ref{l3.9}, we confirm that
\begin{align}\label{e52}
	I_3\sim\oo(t^{-1}). 
\end{align}
For $I_4$, we split the integral into two parts
\begin{align}\label{e53}
	I_4=\int_{\Sigma'}\frac{[C_{w^2}Iw^1](s,n,t)}{2\pi is}\ddddd s+\int_{\Sigma'}\frac{[C_{w^2}(1-C_{w^2})^{-1}C_{w^2}Iw^1](s,n,t)}{2\pi is}\ddddd s;
\end{align}
estimate the first part by Lemma \ref{l3.9},
\begin{align}\label{e55}
	&\Big|\int_{\Sigma'}\frac{[C_{w^2}Iw^1](s,n,t)}{2\pi is}\ddddd s\Big|
	\le\frac{1}{2\pi}\int_{\Sigma_1}\int_{\Sigma_2}\Big|\frac{w^2(s',n,t)w^1(s,n,t)}{s(s'-s)}\Big|\ddddd s'\ddddd s\notag\\
	&\quad\le\frac{1}{2\pi d\inf_{s\in\Sigma_1}|s|}\parallel w_1\parallel_{L^1(\Sigma_1)}\parallel w_2\parallel_{L^1(\Sigma_2)}\lesssim t^{-1};
\end{align}
estimate the second part by Schwartz inequality, the boundedness of $(1-C_{w^2})^{-1}$, Lemma \ref{l3.9} and (\ref{e48})
\begin{align}\label{e56}
	&\Big|\int_{\Sigma'}\frac{[C_{w^2}(1-C_{w^2})^{-1}C_{w^2}Iw^1](s,n,t)}{2\pi is}\ddddd s\Big|\notag\\
	&\quad \le\frac{1}{2\pi}\int_{\Sigma_1}\int_{\Sigma_2}\Big|\frac{[(1-C_{w^2})^{-1}C_{w^2}Iw^2](s',n,t)w^1(s,n,t)}{s(s'-s)}\Big|\ddddd s'\ddddd s\notag\\
	&\quad \le\frac{1}{2\pi d\inf_{s\in\Sigma_1}|s|}\parallel (1-C_{w^2})^{-1}C_{w^2}Iw^2\parallel_{L^1(\Sigma_2)}\parallel w^1\parallel_{L^1(\Sigma_1)}\notag\\
	&\quad\lesssim\parallel C_{w^2}I\parallel_2\parallel w^2\parallel_{L^2(\Sigma_2)}\parallel w^1\parallel_{L^1(\Sigma_1)}\lesssim t^{-1};
\end{align}
thus, by (\ref{e53}), (\ref{e55}) and (\ref{e56}), we have
\begin{align}\label{e57}
	I_4\sim \oo(t^{-1}). 
\end{align}
For $I_5$, we claim that
\begin{align}\label{e58}
	I_5\sim\oo(t^{-1}),
\end{align}
and the technique is parallel to that for $I_4$. 
By (\ref{e44}), (\ref{e52}), (\ref{e57}) and (\ref{e58}), we get
\begin{align}\label{e46}
	M^{(2,R)}(0,n,t)=&I+\sum_{j=1}^{2}\frac{1}{2\pi i}\int_{\Sigma_j}s^{-1}[(1-C_{w^j})^{-1}Iw^j](s,n,t)\ddddd s+\oo(t^{-1}).
\end{align}
\begin{lemma}\label{l3.9}
	For $r\in H^1_\theta(\Sigma)$, we obtain these estimates for $w^j$:
	\begin{align*}
		\parallel w^j\parallel_{L^1(\Sigma^{(2)})}\lesssim t^{-\frac{1}{2}},\quad \parallel w^j\parallel_{L^2(\Sigma^{(2)})}\lesssim t^{-\frac{1}{4}}. 
	\end{align*}
\end{lemma}
\begin{proof}
	Recalling that $\delta$, $\delta^{-1}$, $R_j$ are bounded functions, by (\ref{e27}), (\ref{e28}) and (\ref{e31s}), we obtain that on $\lambda\in\Sigma_j$,
	\begin{align*}
		|w^{j}(\lambda,n,t)|\lesssim e^{-t\sqrt{1-\xi^2}x^2}, \quad x=|\lambda-S_j|,
	\end{align*}
	by which we confirm the result since $w^j$ is supported only on $\Sigma_j$. 
\end{proof}
\begin{lemma}\label{l3.10}
	For $r\in H^1_\theta(\theta)$, we claim that $C_{w^1}C_{w^2}$, $C_{w^1}C_{w^2}$ belong to $\bb(L^2(\Sigma^{(2)}))$ and  $\bb(L^\infty(\Sigma^{(2)}),L^2(\Sigma^{(2)}))$, and we have the following estimate
	\begin{subequations}
		\begin{align}
			&\parallel C_{w^1}C_{w^2}\parallel_{L^2\to L^2}\lesssim t^{-\frac{1}{2}},\quad \parallel C_{w^1}C_{w^2}\parallel_{L^\infty\to L^2}\lesssim t^{-\frac{3}{4}},\label{e41a}\\
			&\parallel C_{w^2}C_{w^1}\parallel_{L^2\to L^2}\lesssim t^{-\frac{1}{2}},\quad \parallel C_{w^2}C_{w^1}\parallel_{L^\infty\to L^2}\lesssim t^{-\frac{3}{4}}.
		\end{align}
	\end{subequations}
	
\end{lemma}
\begin{proof}
	For any $2\times2$ matrix-valued function $f\in L^2(\Sigma^{(2)})$, we obtain that
	\begin{align}\label{e41}
		C_{w^1}C_{w^2}f=C_+(C_-(fw^2_+)w^1_-)+C_-(C_+(fw^2_-)w^1_+).
	\end{align}
	By the formula
	\begin{align*}
		(C_-(fw^2_+)w^1_-)(s_1)=\frac{1}{2\pi i}\int_{\Sigma_2}\frac{f(s_2)w^2_+(s_2)w^1_-(s_1)}{s_2-s_1}\ddddd s_2
	\end{align*}
	ope and Lemma \ref{l3.9}, the $L^2$-norm of the first term in (\ref{e41}) satisfies  that
	\begin{align}\label{e44s}
		&\parallel C_+(C_-(fw^2_+)w^1_-)\parallel_{L^2(\Sigma^{(2)})}\le \frac{1}{2\pi}\left( \int_{\Sigma_1}\Big|\int_{\Sigma_2}\frac{f(s_2)w^2_+(s_2)w^1_-(s_1)}{s_2-s_1}\ddddd s_2\Big|^2\ddddd s_1\right)^{\frac{1}{2}}\notag\\
		&\quad\le\frac{1}{2\pi d}\parallel fw^2_+\parallel_{L^1(\Sigma^{(2)})}\parallel w^1_-\parallel_{L^2(\Sigma^{(2)})}\notag\\
		&\quad\lesssim\begin{cases}
			\parallel f\parallel_{L^\infty(\Sigma^{(2)})}\parallel w^2_+\parallel_{L^1(\Sigma^{(2)})}\parallel w^1_-\parallel_{L^2(\Sigma^{(2)})}\lesssim t^{-\frac{3}{4}}\parallel f\parallel_{L^\infty(\Sigma^{(2)})},\\
			\parallel f\parallel_{L^2(\Sigma^{(2)})}\parallel w^2_+\parallel_{L^2(\Sigma^{(2)})}\parallel w^1_-\parallel_{L^2(\Sigma^{(2)})}\lesssim t^{-\frac{1}{2}}\parallel f\parallel_{L^2(\Sigma^{(2)})},
		\end{cases}
	\end{align}
where $d$ denotes the distance of $\Sigma_1$ and $\Sigma_2$. 
For the second term on the right of (\ref{e41}), we obtain the similar estimate,
\begin{align}\label{e45s}
	\parallel C_+(C_-(fw^2_+)w^1_-)\parallel_{L^2(\Sigma^{(2)})}\lesssim\begin{cases}
		t^{-\frac{3}{4}}\parallel f\parallel_{L^\infty(\Sigma^{(2)})},\\
		t^{-\frac{1}{2}}\parallel f\parallel_{L^2(\Sigma^{(2)})}.
	\end{cases}
\end{align}
By (\ref{e41}), (\ref{e44s}) and (\ref{e45s}), we confirm the result of $C_{w^1}C_{w^2}$, and that of $C_{w^2}C_{w^1}$ is similarly checked.
\end{proof}

\begin{figure}
	\centering\includegraphics[width=0.5\linewidth]{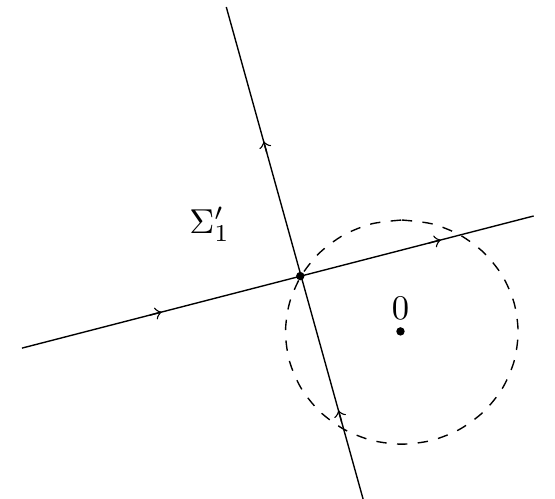}
	\caption{The infinite cross $\Sigma_1'$.}\label{f7}
\end{figure}

Define infinite crosses
\begin{align*}
	\Sigma_j'=S_j((1+e^{\frac{\pi i}{4}}\mathbb{R})\cup(1+e^{-\frac{\pi i}{4}}\mathbb{R}))\quad j=1,2.
\end{align*}
Here, we give the infinite crosses the orientation consisting with that of $\Sigma_j$, example as shown in Figure \ref{f7}.
Set $\check{w}^j_\pm$ the zero extension of $w^j_\pm\big|_{\Sigma_j}$ on $\Sigma_j'$ and $\check{w}^j=\check{w}^j_++\check{w}^j_-$; then, define Cauchy-like integral operators $A_j$: $L^2(\Sigma_j')\to L^2(\Sigma_j')$, $f\mapsto A_jf$,
\begin{align}
	&A_jf=C_+^{\Sigma_j'}(f\check{w}_-)+C_-^{\Sigma_j'}(f\check{w}_+),\label{e51}\\
	&\quad C_\pm^{\Sigma_j'}f(\lambda)=\lim_{\lambda'\to\lambda,\lambda' \text{ on $\pm$ sides of } \Sigma_j'}\int_{\Sigma_j'}\frac{f(s)\ddddd s}{s-\lambda'};\notag
\end{align}
combining (\ref{e46}) and (\ref{e51}), we claim that
\begin{align}\label{e60}
	M^{(2,R)}(0,n,t)=&I+\sum_{j=1}^{2}\frac{1}{2\pi i}\int_{\Sigma_j'}s^{-1}[(1-A_j)^{-1}I\check{w}^j](s,n,t)\ddddd s+\oo(t^{-1}).
\end{align}

\subsection{Scaling and rotation}\label{s3.5}
\indent

Here, we firstly introduce the scaling operator $N_j$ according to each cross, $j=1,2$; 
secondly, we apply these scaling operators to some functions and obtain the oscillatory part $\delta_{j0}$;
thirdly, by these scaling operators, we obtain the Cauchy-like integral operators $\tilde A_j$ and their limits as $t\to+\infty$: $\tilde A_j^\infty$;
finally, by proper evaluations and estimates, we rewrite formula (\ref{e60}) by $\tilde A_j^\infty$, which results in a long-time asymptotic formula (\ref{e79s}) related to $M^{(2,R)}$. 
From (\ref{e79s}), we see that the leading term in the long-time asymptotic formula is determined by the model RH problem at the stationary phase point.
The model RH problem is shown in the later part.

Define infinite crosses
	\begin{align*}
		 \tilde\Sigma_j=e^{\frac{\pi i}{4}}\mathbb{R}\cup e^{-\frac{\pi i}{4}}\mathbb{R},\quad j=1,2,
	\end{align*}
and scaling mapping 
\begin{align}\label{e61}
	N_j: f(\lambda)\mapsto N_jf(\zeta)=f(\beta_j\zeta+S_j). 
\end{align}
It follows from direct computation that $N_j$ belongs to $\bb(L^2(\Sigma_j'),L^2(\tilde\Sigma_j))$ and invertible; moreover, its norm satisfies that
\begin{align*}
	\parallel N_j\parallel_{L^2\to L^2}=|\beta_j|^{-\frac{1}{2}}=t^\frac{1}{4}(1-\xi^2)^\frac{1}{8}.
\end{align*}
Seeing the definition of $\Sigma_j'$ and $\tilde\Sigma_j$, we have the relationship
\begin{align*}
	\Sigma_j'=\beta_j\tilde\Sigma_j+S_j,
\end{align*}
and also set the orientation of $\tilde\Sigma_j$ consisting with that of $\Sigma_j'$ by mapping: $\tilde\Sigma_j\to\Sigma_j'$, $\zeta\mapsto\beta_j\zeta+S_j$.  

By direct computation, we obtain that
\begin{subequations}\label{e64}
	\begin{align}
		&N_j(\delta^2 e^{-it\phi})(\zeta)=\zeta^{(-1)^{j-1}2i\nu_j}e^{(-1)^{j}\frac{i}{2}\zeta^2}\delta^2_{j0}\delta_{j1}(\zeta),\\
		&\quad\delta_{j0}=e^{\alpha_j(S_j)-\frac{it}{2}\phi(S_j)}\left(\frac{(-1)^{j-1}\beta_j}{S_1-S_2}\right)^{(-1)^{j-1}i\nu_j},\\
		&\quad \delta_{j1}(\zeta)=\left(\frac{S_1-S_2}{(-1)^{j-1}\beta_1\zeta+S_1-S_2}\right)^{2(-1)^{j-1}i\nu_1}\notag\\
		&\quad\quad \times e^{2(\alpha_j(\beta_j\zeta+S_j)-\alpha_j(S_j))-it(\phi(\beta_j\zeta+S_j,n,t)-\phi(S_j,n,t))+(-1)^{j-1}\frac{i}{2}\zeta^2},\label{e65c}
	\end{align}
\end{subequations}
where $\delta_{j0}$ is the oscillatory part and $\delta_{j1}(\zeta)$ admits limits at $\zeta\to0$ as shown in Lemma \ref{l3.11}. 
\begin{lemma}\label{l3.11}
	For $r\in H^1_\theta(\Sigma)$, we have $\delta_{11}(\zeta)\to 1$ as $\zeta\to0$, and on $\zeta\in\beta_j^{-1}(\Sigma_j-S_j)$
	\begin{align}\label{e65}
		|\delta_{j1}(\zeta)-1|\lesssim t^{-\frac{1}{4}}|\zeta|^\frac{1}{2},\quad |\delta_{j1}^{-1}(\zeta)-1|\lesssim t^{-\frac{1}{4}}|\zeta|^\frac{1}{2}.
	\end{align}
\end{lemma}
\begin{proof}
	We first detail the proof for the first formula in (\ref{e65}) with $j=1$ and $\zeta\in\beta_1^{-1}(\Sigma_{11}-S_1)$. 
	Since $\left(\frac{S_1-S_2}{\lambda-S_2}\right)^{2i\nu_1}$ and $it(\phi(\lambda,n,t)-\phi(S_1,n,t))-\frac{\phi''(S_j,n,t)}{2}(\lambda-S_j)^2$ are both holomorphic on $\lambda\in\Sigma_{11}$, we deduce that
	\begin{subequations}\label{e59}
	\begin{align}
		&\Big|\left(\frac{S_1-S_2}{\lambda-S_2}\right)^{2i\nu_1}-1\Big|\lesssim |\lambda-S_1|,\\
		&|it(\phi(\lambda,n,t)-\phi(S_1,n,t))-2\phi''(S_1,n,t)(\lambda-S_1)^2|\lesssim t|\lambda-S_1|^3;
	\end{align}
	moreover, seeing (\ref{e17}) and (\ref{e19}), we derive that on $\lambda\in\Sigma_{11}$,
	\begin{align}
		|\alpha_1(\lambda)-\alpha_1(S_1)|\lesssim |\lambda-S_1|^{\frac{1}{2}};
	\end{align}
\end{subequations}
	thus, by (\ref{e61}), (\ref{e65c}) and (\ref{e59}), we obtain that
	\begin{align*}
		|\delta_{11}(\zeta)-1|\lesssim t^{-\frac{1}{4}}|\zeta|^\frac{1}{2},
	\end{align*}
	Since we have proven the first formula in (\ref{e65}) with $j=1$ and $\zeta\in\beta_1^{-1}(\Sigma_{11}-S_1)$, we generalize similarly this proof to the case for $\delta_{j1}(\zeta)$ on $\zeta\in\beta_j^{-1}(\Sigma_j-S_j)$, $j=1,2$. 
	The proof for the second formula is similarly obtained. 
	Thus, we confirm the result in this lemma.
\end{proof}

Define operator $\Delta_{j0}: f\mapsto f\delta_{j0}^{\sigma_3}$, where $f$ is a $2\times2$ matrix-valued function; by the definition of $\delta_{j0}$, it and its inverse operator $\Delta_{j0}^{-1}$ both belong to $\bb(L^2(\tilde\Sigma_j))$. 
Define a Cauchy-like integral operator $\tilde{A}_j$: $L^2(\tilde\Sigma_j)\to L^2(\tilde\Sigma_j)$, $f\mapsto \tilde A_jf$,
\begin{align}
	&\tilde A_jf=C_+^{\tilde\Sigma_j}(f\delta_{j0}^{-\sigma_3} N_j\check{w}^j_-\delta_{j0}^{\sigma_3})+C_-^{\tilde\Sigma_j}(f\delta_{j0}^{-\sigma_3} N_j\check{w}^j_+\delta_{j0}^{\sigma_3}),\label{e54}\\
	&\quad C_\pm^{\tilde\Sigma_j}f(\lambda)=\lim_{\lambda'\to\lambda,\lambda' \text{ on $\pm$ sides of } \tilde\Sigma_j}\int_{\tilde\Sigma_j}\frac{f(s)\ddddd s}{s-\lambda'}.\notag
\end{align}
By the relations (\ref{e51}) and (\ref{e54}), we have these formulas
\begin{align}\label{e67s}
	\tilde A_j=\Delta_{j0}N_jA_jN_j^{-1}\Delta_{j0}^{-1},\quad j=1,2.
\end{align}
Write $\tilde w^j=\tilde w_+^j+\tilde w_-^j$ on $\tilde\Sigma_j$,
\begin{align}\label{e66}
	\tilde w_\pm^j\equiv\tilde w_\pm^j(\zeta,n,t)=\delta_{j0}^{-\sigma_3} N_j\check{w}^j_\pm(\zeta,n,t)\delta_{j0}^{\sigma_3}. 
\end{align}
By the definition of $w_\pm^j$ and (\ref{e64}), we deduce that
\begin{align*}
	&\tilde w_+^1(\zeta,n,t)=\begin{cases}
		\left[\begin{matrix}
			0&-N_1R_3(\zeta)\zeta^{2i\nu_j}e^{-\frac{i}{2}\zeta^2}\delta_{11}(\zeta,n,t)\\0&0
		\end{matrix}\right]&\zeta\in e^{\frac{3\pi i}{4}}(0,|\beta_j^{-1}|\epsilon_0),\\
		\left[\begin{matrix}
			0&-N_1R_6(\zeta)\zeta^{2i\nu_j}e^{-\frac{i}{2}\zeta^2}\delta_{11}(\zeta,n,t)\\0&0
		\end{matrix}\right]&\zeta\in e^{\frac{3\pi i}{4}}(-|\beta_j^{-1}|\epsilon_0,0),\\
	\textbf{0}&\text{otherwise},
	\end{cases}\\
	&\tilde w_-^1(\zeta,n,t)=\begin{cases}
		\left[\begin{matrix}
			0&0\\N_1R_1(\zeta)\zeta^{-2i\nu_j}e^{\frac{i}{2}\zeta^2}\delta_{11}^{-1}(\zeta,n,t)&0
		\end{matrix}\right]&\zeta\in e^{\frac{\pi i}{4}}(0,|\beta_j^{-1}|\epsilon_0),\\
		\left[\begin{matrix}
			0&0\\N_1R_4(\zeta)\zeta^{-2i\nu_j}e^{\frac{i}{2}\zeta^2}\delta_{11}^{-1}(\zeta,n,t)&0
		\end{matrix}\right]&\zeta\in e^{\frac{\pi i}{4}}(-|\beta_j^{-1}|\epsilon_0,0),\\
	\textbf{0}&\text{otherwise},
	\end{cases}\\
&\tilde w_+^2(\zeta,n,t)=\begin{cases}
	\left[\begin{matrix}
		0&-N_2R_3(\zeta)\zeta^{-2i\nu_j}e^{\frac{i}{2}\zeta^2}\delta_{21}(\zeta,n,t)\\0&0
	\end{matrix}\right]&\zeta\in e^{\frac{\pi i}{4}}(0,|\beta_j^{-1}|\epsilon_0),\\
	\left[\begin{matrix}
		0&-N_2R_6(\zeta)\zeta^{-2i\nu_j}e^{\frac{i}{2}\zeta^2}\delta_{21}(\zeta,n,t)\\0&0
	\end{matrix}\right]&\zeta\in e^{\frac{\pi i}{4}}(-|\beta_j^{-1}|\epsilon_0,0),\\
\textbf{0}&\text{otherwise},
\end{cases}\\
&\tilde w_-^2(\zeta,n,t)=\begin{cases}
\left[\begin{matrix}
	0&0\\N_2R_1(\zeta)\zeta^{2i\nu_j}e^{-\frac{i}{2}\zeta^2}\delta_{21}^{-1}(\zeta,n,t)&0
\end{matrix}\right]&\zeta\in e^{\frac{3\pi i}{4}}(0,|\beta_j^{-1}|\epsilon_0),\\
\left[\begin{matrix}
	0&0\\N_2R_4(\zeta)\zeta^{2i\nu_j}e^{-\frac{i}{2}\zeta^2}\delta_{21}^{-1}(\zeta,n,t)&0
\end{matrix}\right]&\zeta\in e^{\frac{3\pi i}{4}}(-|\beta_j^{-1}|\epsilon_0,0),\\
\textbf{0}&\text{otherwise}.
\end{cases}
\end{align*}
Taking limits of $\tilde w_\pm^{j}$ and denoting them as $\tilde w_\pm^{j,\infty}\equiv\tilde w_\pm^{j,\infty}(\zeta)$ when $t\to+\infty$, by direct computation, we obtain that
\begin{subequations}\label{e67}
\begin{align}
	&\tilde w_+^{j,\infty}(\zeta)=\begin{cases}
		\textbf{0}&\zeta\in e^{\frac{(2+(-1)^j)\pi i}{4}}\mathbb{R},\\
		\left[\begin{matrix}
			0&-R_3(S_j)\zeta^{2(-1)^{j-1}i\nu_j}e^{(-1)^j\frac{i}{2}\zeta^2}\\0&0
		\end{matrix}\right]&\zeta\in e^{\frac{(2-(-1)^j)\pi i}{4}}\mathbb{R}^+,\\
		\left[\begin{matrix}
			0&-R_6(S_j)\zeta^{2(-1)^{j-1}i\nu_j}e^{(-1)^j\frac{i}{2}\zeta^2}\\0&0
		\end{matrix}\right]&\zeta\in e^{\frac{(2-(-1)^j)\pi i}{4}}\mathbb{R}^-,
	\end{cases}\\
	&\tilde w_-^{j,\infty}(\zeta)=\begin{cases}
		\textbf{0}&\zeta\in e^{\frac{(2-(-1)^j)\pi i}{4}}\mathbb{R},\\
		\left[\begin{matrix}
			0&0\\R_1(S_j)\zeta^{2(-1)^{j}i\nu_j}e^{(-1)^{j-1}\frac{i}{2}\zeta^2}&0
		\end{matrix}\right]&\zeta\in e^{\frac{(2+(-1)^j)\pi i}{4}}\mathbb{R^+},\\
		\left[\begin{matrix}
			0&0\\R_4(S_j)\zeta^{2(-1)^{j}i\nu_j}e^{(-1)^{j-1}\frac{i}{2}\zeta^2}&0
		\end{matrix}\right]&\zeta\in e^{\frac{(2+(-1)^j)\pi i}{4}}\mathbb{R^-}.
	\end{cases}
\end{align}
\end{subequations}
\begin{figure}
\centering\includegraphics{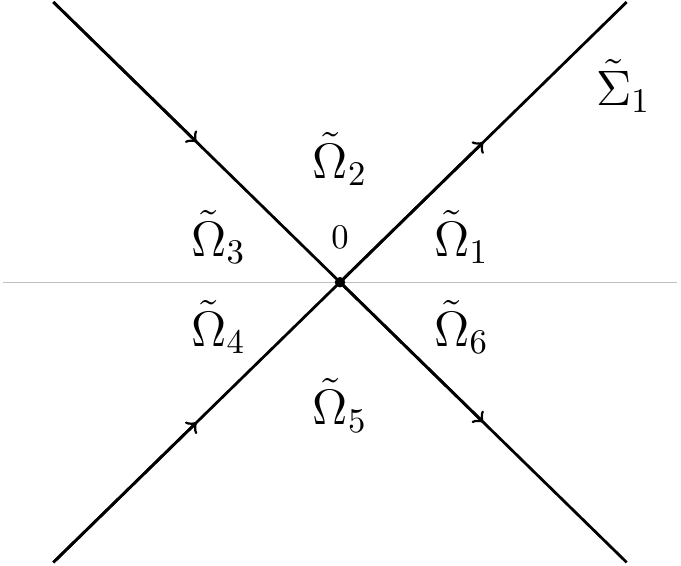}
\caption{The jump contour $\tilde\Sigma_1$ and regions: $\tilde\Omega_1,\dots,\tilde\Omega_6$. }
\end{figure}
\begin{proposition}\label{p3.11}
	If $r\in H^1_\theta(\Sigma)$, we obtain that as $t\to+\infty$, for $j=1,2$,
	\begin{align*}
		\parallel\tilde w^j-\tilde w^{j,\infty}\parallel_{L^\infty(\tilde\Sigma_j)\cap L^2(\tilde\Sigma_j)\cap L^1(\tilde\Sigma_j)}\lesssim t^{-\frac{1}{4}}.
	\end{align*}
	
\end{proposition}
\begin{proof}
	In this proof, we claim that for $\zeta\in\tilde\Sigma_j$, $\tilde w^j$ and $\tilde w^{j,\infty}$ satisfy
	\begin{align}\label{e72}
		&|\tilde w^j-\tilde w^{j,\infty}|(\zeta,n,t)\lesssim 
			t^{-\frac{1}{4}}|\zeta|^\frac{1}{4}e^{-\frac{|\zeta|^2}{2}},
	\end{align}
	and then, this proposition's results naturally follow.
	In the following, we will look into the case of $\zeta\in e^{\frac{\pi i}{4}}\mathbb{R}^+$ and $j=1$; then, proofs for other cases are parallel to omit. 
	For $\zeta=xe^{\frac{\pi i}{4}}\in\beta_j^{-1}(\Sigma_{11}-S_1)\subset e^\frac{\pi i}{4}\mathbb{R}$, we obtain the estimate
	\begin{align}\label{e70}
		&|N_1R_1(\zeta)-R_1(S_1)|=|r(\theta)-r(\theta_1)|=\Big|\int_{\theta_1}^{\theta}r'(\check\theta)\ddddd\check\theta\Big|\le|\theta-\theta_1|^{\frac{1}{2}}\parallel r\parallel_{H^1},\\
		&\quad \quad \theta=\arg(\beta_1\zeta+S_1),\quad \theta_1=\arg S_1\in[0,2\pi];\notag
	\end{align}
seeing the relationship between $\zeta$ and $\theta$ in Figure \ref{f4}, since $\zeta\in\beta_j^{-1}(\Sigma_{11}-S_1)$, we get the following estimate by triangular knowledge
\begin{align}\label{e71}
	|\theta_1-\theta|\lesssim|\beta_1\zeta|= t^{-\frac{1}{2}}(1-\xi^2)^{-\frac{1}{4}}x;
\end{align}
therefore, by (\ref{e70}) and (\ref{e71}), recalling $|\xi|\le V_0$, we obtain that
\begin{align}\label{e73}
	|N_1R_1(\zeta)-R_1(S_1)|\lesssim t^{-\frac{1}{4}}x^{\frac{1}{2}}.
\end{align}
\begin{figure}
	\centering\includegraphics{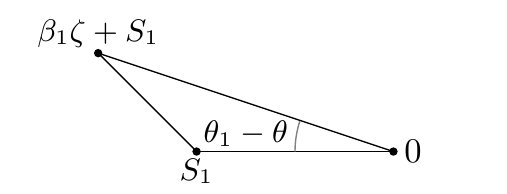}
	\caption{the relation between parameters $\theta$ and $\zeta$ for $\beta_1\zeta+S_1\in\Sigma_{11}$. }\label{f4}
\end{figure}
By definition of $\tilde w^j$, (\ref{e67}), Lemma \ref{l3.11} and (\ref{e73}), we estimate that
	\begin{align*}
		&|\tilde w^j-\tilde w^{j,\infty}|(\zeta,n,t)\le\Big|\zeta^{-2i\nu_1}e^{\frac{i\zeta^2}{2}}\Big|\big|N_1R_1(\zeta)\delta_{11}^{-1}(\zeta)-R_1(S_1)\big|\\
		&\quad\lesssim e^{-\frac{x^2}{2}}(|R_1(\beta_1\zeta+S_1)||\delta_{11}^{-1}(\zeta)-1|+|R_1(\beta_1\zeta+S_1)-R_1(S_1)|)\\
		&\quad\lesssim e^{-\frac{x^2}{2}}(\parallel r\parallel_\infty t^{-\frac{1}{4}}x^{\frac{1}{2}}+t^{-\frac{1}{4}}x^{\frac{1}{2}})\lesssim t^{-\frac{1}{4}}x^{\frac{1}{2}}e^{-\frac{x^2}{2}}.
	\end{align*}
For $\zeta=xe^{\frac{i\pi}{4}}\in e^\frac{\pi i}{4}\mathbb{R}\setminus\beta_j^{-1}(\Sigma_{11}-S_1)$, we obtain by (\ref{e67}) that
\begin{align*}
	&|\tilde w^j-\tilde w^{j,\infty}|(\zeta,n,t)=|\tilde w^{j,\infty}|(\zeta,n,t)=|R_1(S_j)||(xe^{\frac{\pi i}{4}})^{i\nu_1}|e^{-\frac{x^2}{2}}\\&\quad\lesssim e^{-\frac{x^2}{2}}\le\Big|\frac{\beta_1\zeta}{\epsilon_0}\Big|^\frac{1}{2}e^{-\frac{x^2}{2}}\lesssim t^{-\frac{1}{4}}x^\frac{1}{2}e^{-\frac{x^2}{2}}.
\end{align*}
We complete the proof.
\end{proof}
Here, we come to estimates of integral parts in (\ref{e60}). 
By direct computation, (\ref{e67s}), (\ref{e66}) and (\ref{e67}), we obtain that
\begin{align}\label{e75}
	&\frac{1}{2\pi i}\int_{\Sigma_j'}s^{-1}[(1-A_j)^{-1}I\check{w}^j](s,n,t)\ddddd s\notag\\
	&\quad=\frac{1}{2\pi i}\int_{\Sigma_j'}s^{-1}[N_j^{-1}\Delta_{j0}^{-1}(1-\tilde A)^{-1}\Delta_{j0}N_jI\delta_{j0}^{\sigma_3}N_j^{-1}\tilde w^j\delta_{j0}^{-\sigma_3}](s,n,t)\ddddd s\notag\\
	&\quad =\frac{\beta_j}{2\pi i}\int_{\tilde\Sigma_j}(\beta_js+S_j)^{-1}[(1-\tilde A_j)^{-1}\delta_{j0}^{\sigma_3}\tilde w^j](s,n,t)\delta_{j0}^{-\sigma_3}\ddddd s\notag\\
	&\quad=\frac{\beta_j}{2\pi i}\delta_{j0}^{\sigma_3}\int_{\tilde\Sigma_j}(\beta_js+S_j)^{-1}[(1-\tilde A_j)^{-1}I\tilde w^j](s,n,t)\ddddd s\delta_{j0}^{-\sigma_3}\notag\\
	&\quad=\frac{\beta_j}{2\pi i}S_j^{-1}\delta_{j0}^{\sigma_3}\int_{\tilde\Sigma_j}[(1-\tilde A_j^\infty)^{-1}I\tilde w^{j,\infty}](s)\ddddd s\delta_{j0}^{-\sigma_3}+\frac{\beta_j}{2\pi i}\delta_{j0}^{\sigma_3}(I_6+I_7)\delta_{j0}^{-\sigma_3},
\end{align}
where 
\begin{subequations}
	\begin{align}
		&I_6=\int_{\tilde\Sigma_j}(\beta_js+S_j)^{-1}[((1-\tilde A_j)^{-1}-(1-\tilde A_j^\infty)^{-1})I\tilde w^j](s,n,t)\ddddd s,\label{e74a}\\
		&I_7=\int_{\tilde\Sigma_j}[(1-\tilde A_j^\infty)^{-1}I((\beta_js+S_j)^{-1}\tilde w^j-S_j^{-1}\tilde w^{j,\infty})](s,n,t)\ddddd s.
	\end{align}
\end{subequations}
By Schwartz inequality, 
\begin{subequations}
\begin{align}
	&|I_6|=\Big|\int_{\tilde\Sigma_j}[(\beta_js+S_j)^{-1}((1-\tilde A_j)^{-1}(\tilde A_j-\tilde A_j^\infty)(1-\tilde A_j^\infty)^{-1})I\tilde w^j](s,n,t)\ddddd s\Big|\notag\\
	&\quad\le\parallel (1-\tilde A_j)^{-1}(\tilde A_j-\tilde A_j^\infty)(1-\tilde A_j^\infty)^{-1}I\parallel_{L^2(\tilde\Sigma_j)}\parallel (\beta_j\cdot+S_j)^{-1}\tilde w^j\parallel_{L^2(\tilde\Sigma_j)},\label{e75a}\\
	&|I_7|=\int_{\tilde\Sigma_j}[(1-\tilde A_j^\infty)^{-1}I((\beta_js+S_j)^{-1}\tilde w^j-S_j^{-1}\tilde w^{j,\infty})](s,n,t)\ddddd s\notag\\
	&\quad=\int_{\tilde\Sigma_j}[(1-\tilde A_j^\infty)^{-1}\tilde A_jI((\beta_js+S_j)^{-1}\tilde w^j-S_j^{-1}\tilde w^{j,\infty})](s,n,t)\ddddd s\notag\\
	&\quad\quad +\int_{\tilde\Sigma_j}((\beta_j\cdot+S_j)^{-1}
	\tilde w^j-S_j^{-1}\tilde w^{j,\infty})(s,n,t)\ddddd s\notag\\
	&\quad \le\parallel(1-\tilde A_j^\infty)^{-1}\parallel_{L^2\to L^2}\parallel\tilde A_jI\parallel_{L^2(
		\tilde\Sigma_j)}\parallel(\beta_j\cdot+S_j)^{-1}\tilde w^j-S_j^{-1}\tilde w^{j,\infty}\parallel_{L^2(\tilde\Sigma_j)}\notag\\
	&\quad\quad +\parallel(\beta_j\cdot+S_j)^{-1}\tilde w^j-S_j^{-1}\tilde w^{j,\infty}\parallel_{L^1(\tilde\Sigma_j)}.\label{e75b}
\end{align}
\end{subequations}
For $I_6$, estimating by Lemma \ref{l3.11} and $L^2$ boundedness of $\tilde w^{j,\infty}$ on $L^2(\tilde\Sigma_j)$, we get that
\begin{align}\label{e78}
	&\parallel (1-\tilde A_j)^{-1}(\tilde A_j-\tilde A_j^\infty)(1-\tilde A_j^\infty)^{-1}I\parallel_{L^2(\tilde\Sigma_j)}\notag\\
	&\quad=\parallel (1-\tilde A_j)^{-1}(\tilde A_j-\tilde A_j^\infty)I+(1-\tilde A_j)^{-1}(\tilde A_j-\tilde A_j^\infty)(1-\tilde A_j^\infty)^{-1}\tilde A_j^\infty I\parallel_{L^2(\tilde\Sigma_j)}\notag\\
	&\quad\le\parallel(\tilde A_j-\tilde A_j^\infty)I\parallel_{L^2(\tilde\Sigma_j)}+\parallel \tilde A_j-\tilde A_j^\infty\parallel_{L^2\to L^2}\parallel \tilde A_j^\infty I\parallel_{L^2(\tilde\Sigma_j)}\notag\\
	&\quad\lesssim\parallel \tilde w^j-\tilde w^{j,\infty}\parallel_{L^2(\tilde\Sigma_j)}+\parallel \tilde w^j-\tilde w^{j,\infty}\parallel_{L^\infty(\tilde\Sigma_j)}\parallel\tilde w^{j,\infty}\parallel_{L^2(\tilde\Sigma_j)}\notag\\
	&\quad \lesssim t^{-\frac{1}{4}}(1+\parallel\tilde w^{j,\infty}\parallel_{L^2(\tilde\Sigma_j)})\lesssim t^{-\frac{1}{4}}; 
\end{align}
estimate the $L^2$-norm by direct computation, the boundedness of $(\beta_j\zeta+S_j)^{-1}$ on $\zeta\in\tilde\Sigma_j$, (\ref{e61}), (\ref{e66}), Lemma \ref{l3.9} and definition of $\beta_j$
\begin{align}\label{e79}
	&\parallel (\beta_j\cdot+S_j)^{-1}\tilde w^j\parallel_{L^2(\tilde\Sigma_j)}\lesssim \parallel \tilde w^j\parallel_{L^2(\tilde\Sigma_j)}=\parallel \delta_{j0}^{-\sigma_3} N_j\check{w}^j_-\delta_{j0}^{\sigma_3}\parallel_{L^2(\tilde\Sigma_j)}\notag\\
	&\quad\lesssim\parallel N_j\check w^j\parallel_{L^2(\tilde\Sigma_j)}=\parallel w^j\parallel_{L^2(\Sigma_j)}|\beta_j|^{-\frac{1}{2}}\lesssim1;
\end{align}
therefore, we obtain by (\ref{e75a}), (\ref{e78}) and (\ref{e79}) that
\begin{align}\label{e77}
	|I_6|\lesssim t^{-\frac{1}{4}}. 
\end{align}
For $I_7$, similar to Lemma \ref{l3.11}, we obtain that for $j=1,2$,
\begin{align*}
	\parallel(\beta_j\cdot+S_j)^{-1}\tilde w^j-S_j^{-1}\tilde w^{j,\infty}\parallel_{L^1(\tilde\Sigma_j)\cap L^2(\tilde\Sigma_j)}\lesssim t^{-\frac{1}{4}};
\end{align*}
which deduces that
\begin{align}\label{e78s}
	|I_7|\lesssim (\parallel(1-\tilde A_j^\infty)^{-1}\parallel_{L^2\to L^2}\parallel\tilde A_jI\parallel_{L^2(\tilde\Sigma_j)}+1)t^{-\frac{1}{4}}\lesssim t^{-\frac{1}{4}},
\end{align}
by (\ref{e75b}), the boundedness of $(1-\tilde A_j^\infty)^{-1}$ shown in Section \ref{s3.9} and 
\begin{align*}
	\parallel\tilde A_jI\parallel_{L^2(\tilde\Sigma_j)}\lesssim\parallel \tilde w^j\parallel_2\lesssim |\beta_j|^{-\frac{1}{2}}\parallel w^j\parallel_{L^2(\Sigma_j)}\lesssim1.
\end{align*}
Considering (\ref{e60}), (\ref{e75}), (\ref{e77}) and (\ref{e78s}), we obtain the asymptotic formula
\begin{align}\label{e79s}
	M^{(2,R)}(0,n,t)=&I+\sum_{j=1}^{2}\frac{\beta_j}{2\pi i}\delta_{j0}^{\sigma_3}\int_{\tilde\Sigma_j}S_j^{-1}[(1-\tilde A_j^\infty)^{-1}I\tilde w^{j,\infty}](s)\ddddd s\delta_{j0}^{-\sigma_3}+\oo(t^{-\frac{3}{4}}).
\end{align}

\begin{rhp}\label{r3.13}
	Find a $2\times2$ matrix-valued function $M^{L,j}\equiv M^{L,j}(\zeta,n,t)$ such that
	\begin{itemize}
		\item $M^{L,j}(\zeta,n,t)$ is analytic on $\mathbb{C}\setminus\tilde\Sigma_j$.
		\item As $\zeta\to\infty$, $M^{L,j}(\zeta,n,t)\sim I$.
		\item On $\zeta\in\tilde\Sigma_j$, $M^{L,j}_+(\zeta)=M^{L,j}_-(\zeta)V^{L,j}(\zeta)$, where $V^{L,j}=(I-\tilde w^{j,\infty}_-)^{-1}(I+\tilde w^{j,\infty}_+)$.
	\end{itemize}
\end{rhp}
We see that RH problem \ref{r3.13} is a model RH problem that related to parabolic cylinder functions, seeing \cite{deift1993long}. 
And we see that the solution of RH problem \ref{r3.13} is written in the Beals-Coifman solution:
\begin{align}
	M^{L,j}(\zeta)=I+\frac{1}{2\pi i}\int_{\tilde\Sigma_j}\frac{[(1-\tilde A_j^\infty)^{-1}I\tilde w^{j,\infty}](s)}{s-\zeta}\ddddd s
\end{align}
If write $M^{L,j}(\zeta,n,t)$ by its asymptotic expansion at $\zeta\to \infty$,
\begin{align}
	M^{L,j}(\zeta)=I+\zeta^{-1} M^{L,j}_1+\dots,
\end{align}
then, we derive that
\begin{align}
	M^{L,j}_1=\frac{1}{2\pi i}\int_{\tilde\Sigma_j}[(1-\tilde A_j^\infty)^{-1}I\tilde w^{j,\infty}](s)\ddddd s;
\end{align}
recalling (\ref{e79s}), we obtain that
\begin{align}\label{e83}
	M^{(2,R)}(0,n,t)=&I+\sum_{j=1}^{2}\beta_jS_j^{-1}\delta_{j0}^{\sigma_3}	M^{L,j}_1\delta_{j0}^{-\sigma_3}+\oo(t^{-\frac{3}{4}}),
\end{align}
which means that the solution of RH problem \ref{r3.13} describes the leading terms of solutions of $M^{(2,R)}(0,n,t)$. 
Seeing \cite{deift1993long,deift1994long} and solving RH problem \ref{r3.13}, we learn that $M^{L,j}_1$ is explicitly obtain that
\begin{align}\label{e80}
	[M_1^{L,j}]_{1,2}=-i\frac{(2\pi)^\frac{1}{2}e^\frac{\pi i}{4}e^\frac{-\pi\nu_j}{2}}{r(S_j)\Gamma(-i\nu_j)}. 
\end{align}
where $\Gamma(\lambda)$ is Euler's Gamma functions.

\subsection{Analysis on $\bar\partial$-problem}\label{s3.7}
\indent

Since we have obtain the long-time asymptotic formula (\ref{e83}), in this part, we study the solution of $\bar\partial$-problem \ref{d3.6} as $t\to+\infty$. 

The solution of $\bar\partial$-problem \ref{d3.6} can be written as
\begin{align}\label{e84}
	M^{(2,D)}(\lambda,n,t)=I-\frac{1}{\pi}\iint_{\mathbb{C}}\frac{[M^{(2,D)}\tilde\pp](s,n,t)}{s-\lambda}\ddddd s.
\end{align}
(\ref{e84}) can also be written by operator $S$: $f\mapsto Sf$
\begin{align}\label{e85}
	(1-S)M^{(2,D)}=I,\quad Sf(\lambda)=-\frac{1}{\pi}\iint_\mathbb{C}\frac{f(s)\tilde\pp (s)}{s-\lambda}\ddddd L(s),
\end{align}
where $L(s)$ is the Lebesgue distribution on complex plane $\mathbb{C}$. 
We prove in the following that $S\in\bb(L^\infty(\mathbb{C}))$ and as $t$ is sufficiently large, $\parallel S\parallel_{L^\infty\to L^\infty}$ decays.
\begin{proposition}\label{p3.14}
	If $r\in H^1_\theta(\Sigma)$, the operator $S\in\bb(L^\infty(\mathbb{C}))$, and as $t\to+\infty$, the norm of $S$ satisfies
	\begin{align*}
		\parallel S\parallel_{L^\infty\to L^\infty}\lesssim t^{-\frac{1}{4}}.
	\end{align*}
\end{proposition}
\begin{proof}
	As the definition of $\tilde\pp$, it is supported on $\cup_{k=1,3,4,6}\Omega_k$;
	therefore, the integral in (\ref{e85}) is only supported on $\cup_{k=1,3,4,6}\Omega_6$.
	To make the proof concise, we detail the case for functions supported on $\Omega_1^-\deff\Omega_1\cap\{\re\lambda<0\}$;
	and the other cases' proofs are parallel. 
	For $2\times2$ matrix-valued functions $f\equiv f(\lambda)$ supported on $\Omega_1^-$, by (\ref{e24}), (\ref{e25s}) and classical computation, we obtain a series of estimates:
	\begin{align}
		&|Sf(\lambda,n,t)|=\Big|\frac{1}{\pi}\iint_{\Omega_1^-}\frac{f(s)[M^{(2,R)}\bar\partial\pp (M^{(2,R)})^{-1}](s,n,t)}{s-\lambda}\ddddd L(s)\Big|\notag\\
		&\quad\le\parallel f\parallel_{L^\infty(\Omega_1^-)}\parallel M^{(2,R)}\parallel_{L^\infty(\Omega_1^-)}\parallel (M^{(2,R)})^{-1}\parallel_{L^\infty(\Omega_1^-)}\parallel \delta^{-2}\parallel_{L^\infty(\Omega_1^-)}\notag\\
		&\quad\quad\iint_{\Omega_1\cap\{\re\lambda>0\}}\frac{|\bar\partial R_1(s)|e^{-t\im\phi(s,n,t)}\ddddd L(s)}{|s-\lambda|}\notag\\
		&\quad\overset{\circled1}{\lesssim}\parallel f\parallel_{L^\infty(\Omega_1^-)}\iint_{\Omega_1\cap\{\re\lambda>0\}}\frac{|\bar\partial R_1(s)|e^{-t\im\phi(s,n,t)}\ddddd L(s)}{|s-\lambda|}\notag\\
		&\quad\overset{\circled2}{=}\parallel f\parallel_{L^\infty(\Omega_1^-)}\int_{1}^{\rho_0}\int_{\frac{\pi}{2}}^{\theta_\rho}\frac{| r'(\theta)|e^{-t((\rho-\rho^{-1})\sin\theta+2\xi\log\rho)}\rho\ddddd\theta\ddddd\rho}{|\rho e^{i\theta}-\lambda|}\notag\\
		&\quad\overset{\circled3}{\le}\parallel f\parallel_{L^\infty(\Omega_1^-)}\int_{1}^{\rho_0}\int_{\frac{\pi}{2}}^{\theta_\rho}\frac{|r'(\theta)|e^{-t((\rho-\rho^{-1})\sin\theta_\rho+2\xi\log\rho)}\rho_0\ddddd\theta\ddddd\rho}{|\rho e^{i\theta}-\lambda|}\notag\\
		&\quad=\parallel f\parallel_{L^\infty(\Omega_1^-)}\int_{1}^{\rho_0}e^{-t\im\phi(\rho e^{i\theta_\rho},n,t)}\rho_0\ddddd\rho\int_{\frac{\pi}{2}}^{\theta_\rho}\frac{|r'(\theta)|}{|\rho e^{i\theta}-\lambda|}\ddddd\theta\notag\\
		&\quad\overset{\circled4}{\lesssim} \parallel f\parallel_{L^\infty(\Omega_1^-)}\int_{1}^{\rho_0}\frac{e^{-t\im\phi(\rho e^{i\theta_\rho},n,t)}}{|\rho-|\lambda||^\frac{1}{2}}\ddddd\rho
		\overset{\circled5}{\le}\parallel f\parallel_{L^\infty(\Omega_1^-)}\int_{1}^{\rho_0}\frac{e^{-\frac{t\sqrt{1-\xi^2}}{2}(\rho-1)^2}}{|\rho-|\lambda||^\frac{1}{2}}\ddddd\rho\notag\\
		&\quad\overset{\circled6}{\lesssim} t^{-\frac{1}{4}}\parallel f\parallel_{L^\infty(\Omega_1^-)}.\label{e88}
	\end{align}
	In (\ref{e88}),  recalling the boundedness of $\parallel M^{(2,R)}\parallel_{L^\infty(\Omega_1^-)}$,  $\parallel (M^{(2,R)})^{-1}\parallel_{L^\infty(\Omega_1^-)}$, $\parallel \delta^{-1}\parallel_{L^\infty(\Omega_1^-)}$, we obtain estimate $\circled1$.
	For equality $\circled2$, $\theta_\rho$ is the unique argument such that $\rho e^{i\theta_\rho}\in\Sigma_{11}$;
	$\rho_0$ denotes the modular of $s$ belonging to the arc contained in $L\cap D_+$; 
	moreover, for $s=\rho e^{i\theta}$, we obtain that
	\begin{align*}
		\ddddd L(s)=\rho\ddddd\theta\ddddd\rho,\quad \bar\partial=e^{i\theta}(\partial_\rho+i\rho^{-1}\partial_\theta)/2.
	\end{align*}
	For the triangular function, it follows that
	\begin{align*}
		\sin\theta\ge\sin\theta_\rho,\quad  \theta\in[\frac{\pi}{2},\theta_\rho],
	\end{align*}
and then, as a consequence, we obtain $\circled3$.
$\circled4$ is the consequence of Schwartz inequalities, the boundedness of $\parallel r\parallel_{H^1}$ and the elementary fact that
\begin{align*}
	&\int_{\frac{\pi}{2}}^{\theta_\rho}\frac{1}{|\rho e^{i\theta}-\lambda|^2}\ddddd \theta\le\int_{-\pi}^{\pi}\frac{1}{|\rho e^{i\theta}-\lambda|^2}\ddddd \theta=\int_{-\pi}^\pi\frac{\ddddd\theta}{(\rho-|\lambda|)^2+4\rho|\lambda|\sin^2\frac{\theta}{2}}\\
	&\quad\le\int_{-\pi}^\pi\frac{\ddddd\theta}{(\rho-|\lambda|)^2+\frac{4\rho|\lambda|}{\pi^2}\theta^2}\lesssim|\rho-|\lambda||^{-1}.
\end{align*} 
$\circled5$ is obtained by (\ref{e31s}) and the fact that $\rho e^{i\theta_\rho}\in\Sigma_{11}$. 
We derive $\circled6$ by the proof found in \cite{dieng2008long}. 
We complete the proof.
\end{proof}

Since $\parallel S\parallel_{L^\infty\to L^\infty}$ decays as $t\to+\infty$, for sufficiently large $t$, $(1-S)^{-1}$ belongs to $\bb(L^\infty(\mathbb C))$: therefore, by (\ref{e85}), we learn that the solution $M^{(2,D)}$ exists  and (\ref{e84}) is written as 
\begin{align*}
	M^{(2,D)}(\lambda,n,t)=I-\frac{1}{\pi}\iint_\mathbb{C}\frac{[(1-S)^{-1}I\pp](s,n,t)}{s-\lambda}\ddddd s. 
\end{align*}
Estimate $M^{(2,D)}(\lambda,n,t)$ at $\lambda=0$,
\begin{align}\label{e92}
	M^{(2,D)}(0,n,t)=I-\frac{1}{\pi}\iint_\mathbb{C}\frac{[(1-S)^{-1}I\pp](s,n,t)}{s}\ddddd s.
\end{align}
For sufficiently large $t$, using the technique applied on (\ref{e88}), we derive a series of estimates
\begin{align}\label{e91}
	&\Big|\frac{1}{\pi}\iint_{\Omega_1^-}\frac{[(1-S)^{-1}I\tilde\pp](s,n,t)}{s}\ddddd L(s)\Big|\notag\\
	&\quad\le\frac{\parallel (1-S)^{-1}I\parallel_{L^\infty(\Omega_1^-)}}{\pi}\iint_{\Omega_1^-}\frac{|\tilde\pp|(s,n,t)\ddddd L(s)}{|s|}
	\overset{\circled7}{\lesssim}\iint_{\Omega_1^-}\frac{|\bar\partial\pp|(s,n,t)}{|s|}\ddddd L(s)\notag\\
	&\quad=\int_{1}^{\rho_0}\int_{\frac{\pi}{2}}^{\theta_\rho}\frac{|r'(\theta)|e^{-t\im\phi(\rho e^{i\theta},n,t)}}{|\delta^{2}(\rho e^{i\theta})|}\ddddd\theta\ddddd\rho
	\lesssim\int_{1}^{\rho_0}\int_{\frac{\pi}{2}}^{\theta_\rho}|r'(\theta)|e^{-t\im\phi(\rho e^{i\theta},n,t)}\ddddd\theta\ddddd\rho\notag\\
	&\quad\overset{\circled8}{\le}\int_{1}^{\rho_0}e^{-t\im\phi(\rho e^{i\theta_\rho},n,t)}\int_{\frac{\pi}{2}}^{\theta_\rho}|r'(\theta)|e^{-t(\rho-\rho^{-1})(\theta-\theta_\rho)\frac{\cos\theta_\rho}{2}}\ddddd\theta\ddddd\rho\notag\\
	&\quad\lesssim t^{-\frac{1}{2}}\int_{1}^{\rho_0}\frac{e^{-\frac{t\sqrt{1-\xi^2}}{2}(\rho-1)^2}\ddddd\rho}{((1-\rho^{-2})\rho\cos\theta_\rho)^{\frac{1}{2}}}\overset{\circled9}{\lesssim}t^{-\frac{1}{2}}\int_{1}^{\rho_0}\frac{e^{-\frac{t\sqrt{1-\xi^2}}{2}(\rho-1)^2}\ddddd\rho}{((1-\rho^{-2})\frac{\sqrt{1-V_0^2}}{2})^{\frac{1}{2}}}\lesssim t^{-\frac{3}{4}}.
\end{align}
In addition to proof for (\ref{e91}), the sufficiently large $t$ and Proposition \ref{p3.14} guarantee the boundedness of $\parallel(1-S)^{-1}I\parallel_{L^\infty(\mathbb{C})}$, which combined with the boundedness of $\parallel M^{(2,R)}\parallel_{L^\infty(\Omega_1^-)}$,  $\parallel (M^{(2,R)})^{-1}\parallel_{L^\infty(\Omega_1^-)}$, $\parallel \delta^{-1}\parallel_{L^\infty(\Omega_1^-)}$ deduces $\circled7$;
by the elementary fact that for $\theta\in[\frac{\pi}{2},\theta_\rho]$
\begin{align*}
	\sin\theta\ge\sin\theta_\rho+\frac{\cos\theta_\rho}{2}(\theta-\theta_\rho),
\end{align*}
we obtain $\circled8$;
recalling $\rho e^{i\theta_\rho}\in\Sigma_{11}$ and (\ref{e30}), we obtain that
\begin{align}
	\rho\cos\theta_\rho=\re(\rho e^{i\theta_\rho})>\frac{\sqrt{1-V_0^2}}{2}>0,
\end{align}
which deduces $\circled9$.
We extend the result in (\ref{e91}) to that for the case when integral regions is $\mathbb{C}$
\begin{align}\label{e95}
	\Big|\frac{1}{\pi}\iint_{\mathbb{C}}\frac{[(1-S)^{-1}I\tilde\pp](s,n,t)}{s}\ddddd L(s)\Big|\lesssim t^{-\frac{3}{4}};
\end{align}
therefore, combining (\ref{e92}) and (\ref{e95}), we obtain that as $t\to+\infty$, 
\begin{align}\label{e94}
	M^{(2,D)}\sim I+\oo(t^{-\frac{3}{4}}).
\end{align}

\subsection{the long-time asymptotics}
\indent\

Considering (\ref{e15}), (\ref{e20}), (\ref{e23}), (\ref{e22}), (\ref{e83}), (\ref{e80}) and (\ref{e94}), we obtain that
\begin{align*}
	&q_{n}(t)=\delta^{-1}(0)\sum_{j=1}^{2}\beta_jS_j^{-1}\delta_{j0}^{2}	[M^{L,j}_1]_{1,2}+\oo(t^{-\frac{3}{4}}),\\
	&\quad[M_1^{L,j}]_{1,2}=-i\frac{(2\pi)^\frac{1}{2}e^\frac{\pi i}{4}e^\frac{-\pi\nu_1}{2}}{r(S_j)\Gamma(-a)},\quad j=1,2.
\end{align*}

\subsection{The invertibility of operators}\label{s3.9}
\indent

In this part, we come to accomplish this section by proving the boundedness of some operators. 
We check the boundedness of these operators one-by-one in this order: 
\begin{align*}
	&(1-\tilde A_j^\infty)^{-1}\rightharpoonup (1-\tilde A_j)^{-1}\rightharpoonup (1-A_j)^{-1}\\
	&\quad\rightharpoonup(1-C_{w^j})^{-1}\rightharpoonup(1-C_{w'})^{-1}\rightharpoonup(1-C_w)^{-1},\quad j=1,2.
\end{align*}

By taking $\tilde A_j^\infty$ in place of $C_{w^\infty}$ in Step 7 of \cite{deift1994long}, we obtain that $(1-\tilde A_j^\infty)^{-1}\in\bb(L^2(\tilde\Sigma_j))$.
By Proposition \ref{p3.11}, we learn that
\begin{align*}
	\parallel \tilde A_j-\tilde A_j^\infty\parallel_{L^2\to L^2}\lesssim\parallel \tilde w^j-\tilde w^{j,\infty}\parallel_{L^\infty(\tilde\Sigma_j)}\lesssim t^{-\frac{1}{4}},
\end{align*}
which combined with the boundedness of $(1-\tilde A_j^\infty)^{-1}$ deduces that for sufficiently large $t$,
\begin{align*}
	(1-\tilde A_j)^{-1}=(1-\tilde A_j^\infty)^{-1}(1-(\tilde A_j-\tilde A_j^\infty)(1-\tilde A_j^\infty)^{-1})^{-1}\in\bb(L^2(\tilde\Sigma_j)).
\end{align*}
For $(1-A_j)^{-1}$, by (\ref{e67s}), we obtain htat
\begin{align*}
	(1-A_j)^{-1}=N_j^{-1}\Delta_{j0}^{-1}(1-\tilde A_j)^{-1}\Delta_{j0}N_j\in\bb(L^2(\Sigma_j')),\quad j=1,2,
\end{align*}
because $(1-\tilde A_j)^{-1}\in\bb(L^2(\tilde\Sigma_j))$. 
By Lemma 2.56 in \cite{deift1993steepest}, we learn that since $w^j$ is supported on $\Sigma_j\subset\Sigma^{(2)}\cap\Sigma_j'$, that $(1-C_{w^j})^{-1}\in\bb(L^2(\Sigma^{(2)}))$ are equivalent to that $(1-A_j)^{-1}\in\bb(L^2(\Sigma_j'))$. 
By (\ref{e43}), boundedness of $(1-C_{w^j})^{-1}$ and Lemma \ref{l3.10}, we derive that for sufficiently large $t$,
\begin{align*}
	&(1-C_{w'})^{-1}=(1+C_{w^1}(1-C_{w^1})^{-1}+C_{w^2}(1-C_{w^2})^{-1})\notag\\
	&\quad\times(1-C_{w^2}C_{w^1}(1-C_{w^1})^{-1}-C_{w^1}C_{w^2}(1-C_{w^2})^{-1})^{-1}\in \bb(L^2(\Sigma^{(2)})).
\end{align*}
By direct computation and Lemma \ref{l3.7}, we derive that
\begin{align}\label{e98s}
	\parallel C_w-C_{w'}\parallel_{L^2\to L^2}=\parallel C_{w^r}\parallel_{L^2\to L^2}\lesssim\parallel w^r\parallel_{L^\infty(\Sigma^{(2)})}\lesssim t^{-1}.
\end{align}
For sufficiently large $t$, by boundedness of $(1-C_{w'})^{-1}$ and (\ref{e98s}), it follows that
\begin{align*}
	(1-C_{w})^{-1}=(1-C_{w'})^{-1}(1-(C_w-C_{w'})(1-C_{w'})^{-1})^{-1}\in\bb(L^2(\Sigma^{(2)})). 
\end{align*}
We complete the proof of this part.

\section{Long-time asymptotic analysis on $\xi\ge V_0>1$}\label{s4}
\indent

In this section, we analyzing asymptotic properties for the case when $\frac{n}{2t}=\xi\ge V_0>1$. 
In this case, we set $\epsilon_0=\frac{\sqrt{V_0^2-1}+V_0-1}{2}$.
We see that the jump matrix $V$ admits a upper/lower triangular factorization shown in Jump condition of RH problem.
The exponential part in the lower/upper triangular matrix decays fast as $t\to+\infty$ on the $-/+$ side of the jump contour, respectively, seeing (a) in Figure \ref{f3s}. 
According to this factorization, firstly, we deform the RH problem: $M\rightsquigarrow M^{(3)}$, such that it admits $\bar\partial$-RH problem \ref{d4.1} and the new jump matrix decay to $I$ as $t\to+\infty$;
secondly, we deform the solution for $\bar\partial$-RH problem into the product of solutions for RH problem \ref{r4.2} and $\bar\partial$-problem \ref{d4.3}, and analyze long-time asymptotic properties of these solutions separately. 
Finally, we obtain the asymptotic property (\ref{e113}) by the reconstructed formula. 

\begin{figure}
	\quad\quad\quad\begin{subfigure}{0.4\textwidth}
		\quad\quad\quad\ \ \includegraphics[width=0.5\linewidth]{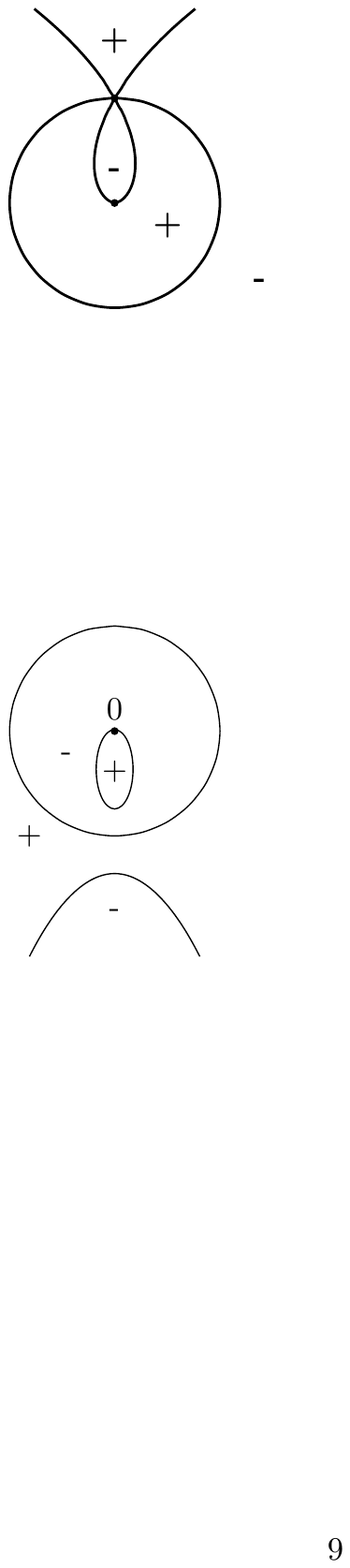}
		\caption{}
	\end{subfigure}
	\begin{subfigure}{0.4\textwidth}
		\quad\quad\quad\ \ \includegraphics[width=0.5\linewidth]{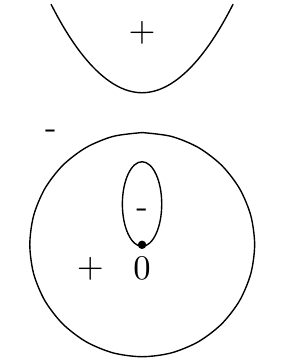}
		\caption{}
	\end{subfigure}
	\caption{The sign of $\im\phi$ in region II and III.}\label{f3s}
\end{figure}

Introducing a $2\times2$ matrix-valued function $M^{(3)}\equiv M^{(3)}(\lambda,n,t)$
\begin{subequations}
\begin{align}
	&M^{(3)}=M\pp^1,\quad\pp^1(\lambda,n,t)=\begin{cases}
		I&\lambda\in\Omega_7\cup\Omega_{10},\\
		\left[\begin{matrix}
			1&0\\R_8(\lambda)e^{it\phi(\lambda,n,t)}&1
		\end{matrix}\right]&\lambda\in\Omega_8,\\
	\left[\begin{matrix}
		1&R_9(\lambda)e^{-it\phi(\lambda,n,t)}\\0&1
	\end{matrix}\right]&\lambda\in\Omega_9,
	\end{cases}\label{e99a}\\
&\quad R_8(\lambda)=-r(\theta),\quad R_9(\lambda)=-\overline{r(\theta)},\quad \lambda=|\lambda| e^{i\theta}, \label{e97b}
\end{align}
\end{subequations}
and
\begin{subequations}\label{e98}
\begin{align}
	&w^{(3)}_+\equiv w^{(3)}_+(\lambda,n,t)=\begin{cases}
		\textbf{0},&\lambda\in\Sigma_{-\epsilon_0}\\
		\left[\begin{matrix}
			0&0\\-R_8(\lambda)e^{it\phi(\lambda,n,t)}&0
		\end{matrix}\right]&\lambda\in\Sigma_{\epsilon_0},
	\end{cases}\\ 
	&w^{(3)}_-\equiv w^{(3)}_-(\lambda,n,t)=\begin{cases}
		\left[\begin{matrix}
			0&R_9(\lambda)e^{-it\phi(\lambda,n,t)}\\0&0
		\end{matrix}\right]&\lambda\in\Sigma_{-\epsilon_0},\\
		\textbf{0}&\lambda\in\Sigma_{\epsilon_0},
	\end{cases}
\end{align}
\end{subequations}
we obtain a $\bar\partial$-RH problem:
\begin{drhp}\label{d4.1}
	Find a $2\times2$ matrix-valued function such that
	\begin{itemize}
		\item $M^{(3)}$ belongs to $C^0(\mathbb{C}\setminus(
		\Sigma_{-\epsilon_0}\cup\Sigma_{\epsilon_0}))$ and its first-order partial derivatives are continuous on $\mathbb{C}\setminus(\Sigma\cup\Sigma_{-\epsilon_0}\cup\Sigma_{\epsilon_0})$.
		\item As $\lambda\to\infty$, $M^{(3)}(\lambda,n,t)\sim I+\oo(\lambda^{-1})$.
		\item On $\lambda\in(\Sigma_{\epsilon_0}\cup\Sigma_{-\epsilon_0})$, $M^{(3)}_+=M^{(3)}_-V^{(3)}$, $ V^{(3)}=(1-w^{(3)}_-)^{-1}(1+w^{(3)}_+)$.
		\item On $\lambda\in\mathbb{C}\setminus(\Sigma\cup\Sigma_{-\epsilon_0}\cup\Sigma_{\epsilon_0})$,
		\begin{align*}
			&\bar\partial M^{(3)}=M^{(3)}\bar\partial\pp^1,\\ &\bar\partial\pp^1(\lambda,n,t)=\begin{cases}
				\textbf{0}&\lambda\in\Omega_7\cup\Omega_{10},\\
				\left[\begin{matrix}
					0&0\\\bar\partial R_8(\lambda)e^{it\phi(\lambda,n,t)}&0
				\end{matrix}\right]&\lambda\in\Omega_8,\\
				\left[\begin{matrix}
					0&\bar\partial R_9(\lambda)e^{-it\phi(\lambda,n,t)}\\0&0
				\end{matrix}\right]&\lambda\in\Omega_9.
			\end{cases}
		\end{align*}
	\end{itemize}
\end{drhp}
\begin{figure}
	\centering\includegraphics[width=0.5\linewidth]{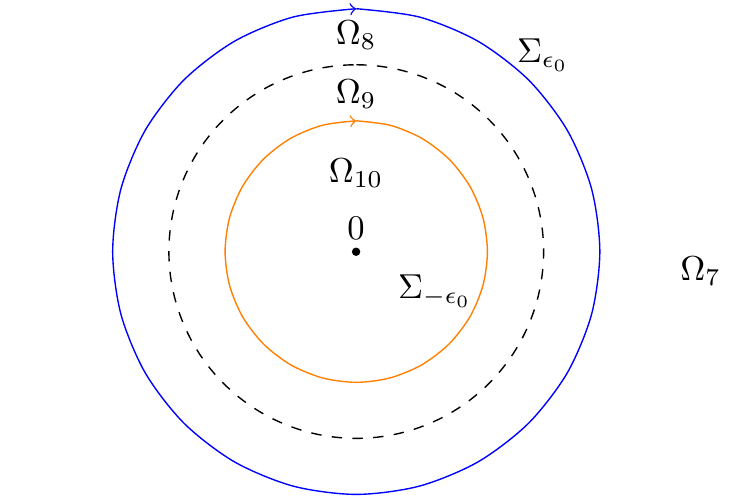}
	\caption{The jump contour: $\Sigma=\Sigma_{-\epsilon_0}\cup\Sigma_{\sigma_0}$.}\label{f8}
\end{figure}
For sufficiently large $t>0$, we deform the solution for $\bar\partial$-RH problem \ref{d4.1} into the product of solutions for RH problem \ref{r4.2} and $\bar\partial$ problem \ref{d4.3}:
\begin{align}\label{e101}
	M^{(3)}=M^{(3,D)}M^{(3,R)},
\end{align}
where $M^{(3,R)}\equiv M^{(3,R)}(\lambda,n,t)$ admits the same jump condition as $M^{(3)}$'s and $M^{(3,D)}\equiv M^{(3,D)}(\lambda,n,t)$ is continuous over complex plane $\mathbb{C}$;
then, in the remaining part of this section, by analyzing long-time asymptotic properties of these solutions, we claim that
\begin{align}\label{e102}
	M^{(3)}(0,n,t)\sim I+\oo(t^{-1}).
\end{align}
Below, we study the following problems.
\begin{rhp}\label{r4.2}
	Find a $2\times2$ matrix-valued function $M^{(3,R)}$ such that
	\begin{itemize}
		\item $M^{(3,R)}$ is holomorphic on $\mathbb{C}\setminus(\Sigma_{-\epsilon_0}\cup\Sigma_{\epsilon_0})$.
		\item As $\lambda\to\infty$, $M^{(3,R)}(\lambda,n,t)\sim I+\oo(\lambda^{-1})$.
		\item On $\lambda\in\Sigma_{-\epsilon_0}\cup\Sigma_{\epsilon_0}$, $M^{(3,R)}_+=M^{(3,R)}_-V^{(3)}$.
	\end{itemize}
\end{rhp}
\begin{dbarproblem}\label{d4.3}
	Find a $2\times2$ matrix-valued function $M^{(3,D)}$ such that
	\begin{itemize}
		\item $M^{(3,D)}$ is continuous on $\mathbb{C}$, and its first-order partial derivatives are continuous on $\mathbb{C}\setminus(\Sigma\cup\Sigma_{-\epsilon_0}\cup\Sigma_{\epsilon_0})$.
		\item As $\lambda\to\infty$, $M^{(3,D)}\sim I+\oo(\lambda^{-1})$.
		\item On $\lambda\in\mathbb{C}\setminus(\Sigma\cup\Sigma_{-\epsilon_0}\cup\Sigma_{\epsilon_0})$, 
		\begin{align*}
			\bar\partial M^{(3,D)}=M^{(3,D)}\tilde\pp^1,\quad \tilde\pp^1=M^{(3,R)}\bar\partial\pp^1(M^{(3,R)})^{-1}.
		\end{align*}
		
	\end{itemize}
\end{dbarproblem}

For RH problem \ref{r4.2}, like what we have done to RH problem \ref{r3.4}, we obtain the Beals-Coifman solution for RH problem \ref{r4.2},
\begin{align}\label{e100}
	&M^{(3,R)}(\lambda,n,t)=I+\frac{1}{2\pi i}\int_{\Sigma^{(3)}}\frac{[(1-C_{w^{(3)}}^{\Sigma^{(3)}})^{-1}Iw^{(3)}](s,n,t)}{s-\lambda}\ddddd s,\\
	&\quad C_{w^{(3)}}^{\Sigma^{(3)}}=C_+^{\Sigma^{(3)}}(\cdot w^{(3)}_-)+C_-^{\Sigma^{(3)}}(\cdot w^{(3)}_+),\quad w^{(3)}=w_+^{(3)}+w^{(3)}_-,\notag\\
	&\quad C_\pm^{\Sigma^{(3)}}f(\lambda)=\lim_{\lambda'\to\lambda,\lambda\text{ on the $\pm$ side of $\Sigma^{(3)}$}}\int_{\Sigma^{(3)}}\frac{f(s)}{2\pi i(s-\lambda')}\ddddd s.\notag
\end{align}
Since $\Sigma_{\epsilon_0}$, $\Sigma_{-\epsilon_0}$ are compact curves and contained in the region $\{\im\phi>0\}$, $\{\im\phi<0\}$, respectively, there is a positive constant $C$ such that
\begin{align}
	\im\phi\big|_{\Sigma_{\epsilon_0}}\ge C,\quad\im\phi\big|_{\Sigma_{-\epsilon_0}}\le-C;
\end{align}
therefore, by the fact that $r\in H^1_\theta(\Sigma)\subset L^\infty(\Sigma)$, (\ref{e97b}) and (\ref{e98}), we obtain that as $t\to+\infty$, 
\begin{align}
	\parallel w^{(3)}_\pm\parallel_{L^\infty(\Sigma^{(3)})}\lesssim e^{-Ct};
\end{align}
as a result, using the technique applied in (\ref{e40}), we deduce that
\begin{align}\label{e103}
	\parallel C^{(\Sigma^{(3)})}_{w^{(3)}}\parallel_{L^2\to L^2}\lesssim\parallel w^{(3)}\parallel_{L^\infty(\Sigma^{(3)})}\lesssim e^{-Ct}.
\end{align}
Considering (\ref{e100}) and (\ref{e103}), we obtain that as $t\to+\infty$, 
\begin{align}\label{e107}
	M^{(3,R)}(0,n,t)\sim I+\oo(e^{-Ct}).
\end{align}
For the boundedness of $M^{(3,R)}$, by RH problem \ref{r4.2}, we find that $\det M^{(3,R)}$ admits no jump on $\mathbb{C}$, and as $\lambda\to\infty$, 
\begin{align*}
	\lim_{\lambda\to\infty}\det M^{(3,R)}\sim 1+\oo(\lambda^{-1});
\end{align*}
therefore, by Liouville's Theorem, we derive that
\begin{align*}
	\det M^{(3,R)}\equiv1;
\end{align*}
moreover, since $M^{(3,R)}$ admits no pole, both $M^{(3,R)}$ and $(M^{(3,R)})^{-1}$ are bounded on $\mathbb{C}$.

For $\bar\partial$-problem \ref{d4.3}, Like that in Section \ref{s3.7}, we also have the solution for $\bar\partial$-problem \ref{d4.3}:
\begin{align}\label{e108}
	M^{(3,D)}(\lambda,n,t)=I+[S^1M^{(3,D)}](\lambda,n,t),
\end{align}
where $S^1$: $f\mapsto S^1f$ is a linear operator and belongs to  $\bb(L^\infty(\mathbb{C}))$,
\begin{align}\label{e109}
	S^1f(\lambda,n,t)=-\frac{1}{\pi}\iint_\mathbb{C}\frac{f(s)\tilde\pp^1 (s,n,t)}{s-\lambda}\ddddd L(s).
\end{align}
For this operator $S^1$, we find it decays as $t\to+\infty$ on $L^\infty(\mathbb{C})$ as shown in Proposition \ref{l4.4}.
\begin{proposition}\label{l4.4}
	For $r\in H^1_\theta(\Sigma)$, the integral operator $S^1$ belongs to $\mathcal{B}(L^\infty(\mathbb{C}))$, and it satisfies that
	\begin{align*}
		\parallel S^1\parallel_{L^\infty\to L^\infty}\lesssim t^{-\frac{1}{4}}.
	\end{align*}
\end{proposition}
\begin{proof}
	We detail the case when functions $f(s)$ are only supported on $\Omega_8$. 
	Like the proof in Proposition \ref{p3.14}, we obtain a series of estimates: 
	\begin{align}\label{e110}
		&|S^1f(\lambda)|\le\frac{\parallel f\parallel_{L^\infty(\Omega_8)}}{\pi}\parallel M^{(3,R)}\parallel_{L^\infty(\Omega_8)}\parallel (M^{(3,R)})^{-1}\parallel_{L^\infty(\Omega_8)}\notag\\
		&\quad\quad\iint_{\Omega_8}\frac{|\bar\partial R_8(s)|e^{-t\im\phi(s,n,t)}}{|s-\lambda|}\rho\ddddd L(s)\notag\\
		&\quad\overset{\circled{10}}{\lesssim}\parallel f\parallel_{L^\infty(\Omega_8)}\int_1^{1+\epsilon_0}\int_{0}^{2\pi}\frac{|\bar\partial R_8(s)|e^{-t(2\xi\ln\rho-(\rho-\rho^{-1}))}}{|\rho e^{i\theta}-\lambda|}\rho\ddddd\theta\ddddd\rho\notag\\
		&\quad=\parallel f\parallel_{L^\infty(\Omega_8)}\int_1^{1+\epsilon_0}\Big| \frac{r'(\theta)}{\rho e^{i\theta}-\lambda}\Big|\ddddd\theta\int_{0}^{2\pi}e^{-t(2\xi\ln\rho-(\rho-\rho^{-1}))}\rho\ddddd\rho\notag\\
		&\quad\overset{\circled{11}}{\le}\parallel f\parallel_{L^\infty(\Omega_8)}\parallel r'\parallel_{L^2_\theta(\sigma)}\int_{0}^{2\pi}\parallel\frac{1}{\rho e^{i\theta}-\lambda}\parallel_{L^2(\{|s|=\rho\})}e^{-2t(V_0\frac{\ln(1+\epsilon_0)}{\epsilon_0}-1)(\rho-1)}\rho\ddddd\rho\notag\\
		&\quad\lesssim\parallel f\parallel_{L^\infty(\Omega_8)} \int_{0}^{2\pi}\frac{e^{-2t(V_0\frac{\ln(1+\epsilon_0)}{\epsilon_0}-1)(\rho-1)}}{|\rho-|\lambda||^{\frac{1}{2}}}\rho\ddddd\rho\lesssim t^{-\frac{1}{4}}\parallel f\parallel_{L^\infty(\Omega_8)}.
	\end{align}
	In (\ref{e110}), by the fact that for $s=\rho e^{i\theta}$, 
	\begin{align*}
		\im\phi(s,n,t)=(\rho-\rho^{-1})\sin\theta+2\xi\ln\rho\ge 2\xi\ln\rho-(\rho-\rho^{-1}),
	\end{align*}
	which with the boundedness of $\parallel M^{(3,R)}\parallel_{L^\infty(\Omega_8)}$ and $\parallel (M^{(3,R)})^{-1}\parallel_{L^\infty(\Omega_8)}$ deduces $\circled{10}$; 
	$\circled{11}$ is by Schwartz Inequalities and the fact that on the region $\Omega_8$, $\rho>1$ such that
	\begin{align*}
		2\xi\ln\rho-\rho+\rho^{-1}=\left(\frac{2\xi\ln\rho}{\rho-1}-\left(1+\frac{1}{\rho}\right)\right)(\rho-1)\ge2(\rho-1)\left(\frac{V_0\ln(1+\epsilon_0)}{\epsilon_0}-1\right)
	\end{align*}
	where $\frac{V_0\ln(1+\epsilon_0)}{\epsilon_0}-1$ is a positive number because of the choosing of $\epsilon_0$ and $V_0>1$.
	We complete the proof for $f(s)$ only supported on $\Omega_8$, and the proof for general $f(s)\in L^\infty(\mathbb{C})$ is similarly obtained. 
\end{proof}
From Proposition \ref{l4.4}, we see that for sufficiently large $t$, $(1-S^1)^{-1}I$ exists and is bounded on $L^\infty(\mathbb{C})$;
therefore, we obtain that
\begin{align}\label{e111}
	&\Big|\frac{1}{\pi}\iint_{\Omega_8}\frac{[(1-S)^{-1}I\tilde\pp^1](s,n,t)}{s}\ddddd L(s)\Big|\notag\\
	&\quad\le\frac{1}{\pi}\parallel (1-S)^{-1}I\parallel_{L^\infty(\mathbb{C})}\parallel M^{(3,R)}\parallel_{L^\infty(\Omega_8)}\parallel (M^{(3,R)})^{-1}\parallel_{L^\infty(\Omega_8)}\notag\\
	&\quad\quad\iint_{\Omega_8}\frac{|\bar\partial R_8(s)|e^{-t\im\phi(s,n,t)}}{|s|}\rho\ddddd L(s)\notag\\
	&\quad\lesssim\iint_{\Omega_8}\frac{|\bar\partial R_8(s)|e^{-t\im\phi(s,n,t)}}{|s|}\rho\ddddd L(s)=\int_{1}^{1+\epsilon_0}\int_{0}^{2\pi}\frac{|r'(\theta)|\ddddd\theta\ddddd\rho}{e^{t((\rho-\frac{1}{\rho})\sin\theta+2\xi\ln\rho)}}\notag\\
	&\quad\le\parallel r'\parallel_{L^1(\Sigma)}\int_{1}^{1+\epsilon_0}e^{-t(2\xi\ln\rho-\rho+\rho^{-1})}\ddddd\rho\notag\\
	&\quad \le\parallel r\parallel_{L^1(\Sigma)}\int_{1}^{1+\epsilon_0}e^{-2t(V_0\frac{\ln(1+\epsilon_0)}{\epsilon_0}-1)(\rho-1)}\ddddd\rho\lesssim t^{-1};
\end{align}
similar to (\ref{e111}), we obtain the result for integral region $\mathbb{C}$ that
\begin{align}\label{e112}
	\Big|\frac{1}{\pi}\iint_{\mathbb{C}}\frac{[(1-S)^{-1}I\tilde\pp^1](s,n,t)}{s}\ddddd L(s)\Big|\lesssim t^{-1}.
\end{align}

Considering (\ref{e101}), (\ref{e107}), (\ref{e108}), (\ref{e109}) and (\ref{e112}), we prove (\ref{e102});
therefore, by (\ref{e15}), (\ref{e99a}) and (\ref{e102}), we obtain that
\begin{align*}
	q_n(t)\sim \oo(t^{-1}).
\end{align*}

\section{Long-time asymptotic analysis on $\xi\le-V_0<-1$}\label{s5}
\indent

In the last section, we analyze the case of $\xi\le-V_0<-1$. 
Introducing a $2\times2$ matrix-valued function
\begin{align}\label{e113}
	M^{(4)}=M\tilde\delta^{-\sigma_3},\quad\tilde\delta\equiv\tilde\delta(\lambda)=e^{\frac{1}{2\pi i}\int_\Sigma\frac{\ln(1-|r(s)|^2)\ddddd s}{s-\lambda}}
\end{align}
we claim that $\tilde\delta$ satisfies Proposition \ref{p5.1};
moreover, by Proposition \ref{p5.1} and RH problem \ref{r2.1}, we obtain that $M^{(4)}$ admits RH problem \ref{r5.2}. 
We see in RH problem \ref{r5.2} that on $\lambda\in\Sigma$, the jump matrix admits a lower/upper triangular factorization
\begin{align}
	&V^{(4)}=(I-w^{(4)}_-)^{-1}(I-w^{(4)}_+),\\
	&\quad w^{(4)}_-\equiv w^{(4)}_-(\lambda,n,t)=\left[\begin{matrix}
		0&0\\
		\frac{\delta_-^{-2}(\lambda)r(\lambda)}{1-|r(\lambda)|^2}e^{it\phi(\lambda,n,t)}&0
	\end{matrix}\right],\notag\\
&\quad w^{(4)}_+\equiv w^{(4)}_+(\lambda,n,t)=\left[\begin{matrix}
		0&-\frac{\delta_+^2(\lambda)\overline{r(\lambda)}}{1-|r(\lambda)|^2}e^{-it\phi(\lambda,n,t)}\\
		0&0
	\end{matrix}\right],\notag
\end{align}
and the exponential part of $w^{(4)}_+$, $w^{(4)}_-$ decay exponentially on regions to the $+$, $-$ side of $\Sigma$ as $t\to+\infty$, respectively, seeing (b) in Figure \ref{f3s};
therefore, similar to the case of $\xi\ge V_0>1$ we obtain the long-time asymptotics when $\xi\le-V_0<-1$, and that
\begin{align*}
	q_n(t)\sim \oo(t^{-1}).
\end{align*}
\begin{proposition}\label{p5.1}
	Since $r\in H^1_\theta(\Sigma), $the scalar function $\tilde\delta$ satisfies: 
	\begin{enumerate}[label=(\alph*)]
		\item 
		$\tilde\delta$ is analytic on $\mathbb{C}\setminus \Sigma$.
		\item 
		As $\lambda\to\infty$,
		\begin{align*}
			\tilde\delta(\lambda)\sim 1+O(\lambda^{-1}).
		\end{align*}
		\item On the unit circle $\lambda\in\Sigma$,
		\begin{align*}
			\tilde\delta_+(\lambda)=\tilde\delta_-(\lambda)(1-|r(\lambda|^2)).
		\end{align*}
		\item $\tilde\delta$ admits the symmetry:
		\begin{align*}
			\tilde\delta(\lambda)=\overline{\tilde\delta(0)/\tilde\delta(\bar\lambda^{-1})}.
		\end{align*}
	\end{enumerate}
\end{proposition}
\begin{proof}
	This proof is parallel to that for Proposition \ref{p3.1}.
\end{proof}
\begin{rhp}\label{r5.2}
	Find a $2\times2$ matrix-valued function $M^{(4)}$ such that
	\begin{itemize}
		\item \textbf{Analyticity}: $M^{(1)}$ is analytic on $\mathbb{C}\setminus\Sigma$.
		\item \textbf{Normalization}: As $\lambda\to\infty$, 
		\begin{align*}
			M^{(4)}(\lambda,n,t)\sim I+O(\lambda^{-1}).
		\end{align*}
		\item \textbf{Jump condition}: On $\lambda\in\Sigma$,
		\begin{align*}
			&M^{(4)}_+=M^{(4)}_-V^{(4)},\\
			&\quad V^{(4)}\equiv V^{(3)}(\lambda,n,t)=\left[\begin{matrix}
				1&-\frac{\tilde\delta_+^2(\lambda)\overline{r(\lambda)}}{1-|r(\lambda)|^2}e^{-it\phi(\lambda,n,t)}\\
				\frac{\tilde\delta_-^{-2}(\lambda)r(\lambda)}{1-|r(\lambda)|^2}e^{it\phi(\lambda,n,t)}&1-|r(\lambda)|^2
			\end{matrix}\right].
		\end{align*}
	\end{itemize}
\end{rhp}

\bibliography{references}
\bibliographystyle{unsrt}

\end{document}